\documentclass[11pt]{article}
\usepackage{amsfonts,latexsym}
\usepackage{amsmath}
\usepackage{amsthm}
\usepackage{amssymb}
\usepackage{mathrsfs}
\usepackage{color}
\usepackage{cmll}
\usepackage{bm}
\usepackage{enumerate}
\usepackage{extarrows}
\usepackage{indentfirst}
\usepackage{wasysym}
\usepackage{graphicx}
\usepackage{lscape}
\usepackage{dsfont}
\usepackage[header,title,page,titletoc]{appendix}
\usepackage{cite}
\usepackage{xcolor}
\usepackage[pagewise]{lineno}
\usepackage{hyperref}
\hypersetup{hypertex=true, colorlinks=true, linkcolor=red, anchorcolor=red, citecolor=red}
\usepackage{caption}
\usepackage{subfigure}
\usepackage{scalefnt}
\usepackage{tikz}
\usetikzlibrary{arrows,shapes,chains}

\newtheorem{theorem}{\indent Theorem}[section]
\newtheorem{proposition}[theorem]{\indent Proposition}
\newtheorem{definition}[theorem]{\indent Definition}
\newtheorem{lemma}[theorem]{\indent Lemma}
\newtheorem{assumption}{Assumption}[section]
\newtheorem{remark}[theorem]{\indent Remark}
\newtheorem{example}[theorem]{\indent Example}
\newtheorem{corollary}[theorem]{\indent Corollary}

\begin{document}
\title
{Dynamics of the quintic wave equation with nonlocal weak damping}
\author{Feng Zhou$^{1,*}$, Hongfang Li$^{1}$, Kaixuan Zhu$^{2}$, Xinyu Mei$^{3}$
\\
{\small \it $^{1}$College of Science, China University of Petroleum (East China),}
\\
{\small \it Qingdao, 266580, P.R. China}
\\
{\small \it $^{2}$School of Mathematics and Physics Science, Hunan University of Arts and Science,}\\
{\small \it Changde, 415000, P.R. China}\\
{\small \it $^{3}$School of Mathematics and Statistics, Yunnan University,}\\
{\small \it Kunming, 650091, China}
}
\date{}

\renewcommand{\theequation}{\arabic{section}.\arabic{equation}}
\numberwithin{equation}{section}

\maketitle

\begin{abstract}
%

This article presents a new scheme for studying the dynamics of a quintic wave equation with nonlocal weak damping in a 3D smooth bounded domain. As an application, the existence and structure of weak, strong, and exponential attractors for the solution semigroup of this equation are obtained. The investigation sheds light on the well-posedness and long-time behaviour of nonlinear dissipative evolution equations with nonlinear damping and critical nonlinearity.

\end{abstract}


\footnote[0]{\hspace{-7.4mm}
$^*$Corresponding author: zhoufeng13@upc.edu.cn
\\
\\{\small \it E-mail addresses: \rm lihongfang@upc.edu.cn, zhukx12@163.com, meixy@ynu.edu.cn.}
}

\vspace{-1 cm}
\section{Introduction}
In this paper, we are concerned with the following wave model with a nonlocal weak damping:
\begin{equation}\label{1.1}
\begin{cases}
\partial_{t}^{2}u+Au+\mathcal{J}(\|\partial_{t}u(t)\|^{2})\partial_{t}u
+g(u)=
h(x),x\in \Omega,
\\ u|_{\partial\Omega}=0,
\\u(x,0)=u^{0}, ~\partial_{t}u(x,0)=u^{1},
\end{cases}
\end{equation}
where $\Omega\subset\mathbb{R}^{3}$ is a bounded smooth domain, $A=-\Delta$, $\|\cdot\|$ is the norm in $L^{2}(\Omega)$, $\mathcal{J}(\cdot)$ is a scalar function, $g(u)$ is a given source term and $h\in L^{2}(\Omega)$ is an external force term.


Weakly damped semilinear wave equations are used to model a wide range of oscillatory processes in various fields, including physics, engineering, biology and geoscience. The wave equations with various nonlocal damping forces have been the subject of extensive investigation by several authors in recent years. These include the nonlocal fractional damping $J(\|\nabla u\|^{2})(-\Delta)^{\theta}\partial_{t}u$ $(\frac{1}{2}\leq\theta<1)$, nonlocal strong damping $J(\|\nabla u\|^{2})(-\Delta)\partial_{t}u$ and nonlocal nonlinear damping $J(\|\nabla u\|^{2})g(\partial_{t}u)$, see \cite{chueshov,chueshov00,yangli,sunz} for more details. It is noteworthy that the nonlocal damping coefficients presented in the aforementioned papers are functions of the $L^{2}$--norm of the gradient of the displacement.


In 2013, Haraux $et$ $al$ in \cite{aloui} provided an illustrative example of the wave equation
\begin{align}\label{E1.2}
\partial_{t}^{2}u-\Delta u+\left(\int_{\Omega}\left|\partial_{t}u\right|^2dx\right)^{\frac{\alpha}{2}}\partial_{t}u=h(t,x)
\end{align}
with a nonlinear damping term depending on a power of the norm of the velocity and they established compactness properties of trajectories to the equation \eqref{E1.2} under suitable conditions. In 2014, motivated by the works of J\"{o}rgens \cite{jorgens} and Schiff \cite{schiff} on a nonlinear theory of meson fields, Lour\^{e}do $et$ $al$ \cite{louredo} introduced the term $g\left(\int_{\Omega}\left|\partial_{t}u\right|^2 d x\right)\partial_{t}u$ to describe an internal dissipation mechanism, where $u$ describing the meson field amplitude. Subsequently, the study of the wave equation with nonlocal weak damping $g\left(\int_{\Omega}\left|\partial_{t}u\right|^2dx\right)\partial_{t}u$ which is also called averaged damping gained considerable attention. For example, the wave equation like \eqref{1.1} with nonlinear boundary damping or with nonlocal nonlinear source terms was studied by Zhang $et$ $al$ in \cite{zhw1} and \cite{zhw}; Zhong $et$ $al$ made some significant progress in the dynamics of the wave equations \eqref{1.1} with a nonlocal nonlinear damping term $\|\partial_{t}u\|^{p}\partial_{t}u$ in recent years, see \cite{zp,zhong2,zhongyanzhu,zhong,zzt,zhong1,zzhaoz,zz} for more details.

It is important to note that the long-term dynamics of Eq. \eqref{1.1} depends strongly on the growth rate $q$ of the non-linearity $g$ with $g(u)\sim |u|^{q-1}u$. Historically, the growth exponent $q=3$ has been considered as a critical exponent for the case of 3D bounded domain and there exists a huge literature on the well-posedness and long-term dynamics of wave equations with $q<3$ and $q=3$, see \cite{ball,caraballo,temam,chueshov1} and references therein. Thus, it seems natural to extend these results to the sup-cubic case. However, in the supercritical case $q>4$, the global well-posedness of Eq. \eqref{1.1} is still an open problem. 

We now focus on the intriguing case where $3< q\leq 5$. In this scenario, the uniqueness of energy weak solutions remains an open question, typically addressed using Sobolev inequality techniques. To tackle this challenge, Strichartz estimates for the solutions of Eq. \eqref{1.1}, such as $u\in L_{loc}^{4}(\mathbb{R},L^{12}(\Omega))$, prove to be effective. Solutions satisfying these estimates are referred to as Shatah--Struwe (S--S) solutions (see \cite{ksz}). By employing appropriate versions of Strichartz estimates and the Morawetz--Pohozaev identity in bounded domains, one can establish the global well-posedness of S--S solutions (see \cite{bss, blp, bp}). The dynamics of S--S solutions in the context of weakly damped wave equations, where $\mathcal{J}(\cdot)\equiv$ const $>0$ or the damping coefficients explicitly depend on time, have been extensively explored (see \cite{chang,ksz,mei0,mssz,lms,sav,sav1,sav2,sunxiong} for a comprehensive survey). In 2023, Zhong $et$ $al$  demonstrated the existence of a uniform polynomial attractor for Eq. \eqref{1.1} with $\mathcal{J}(s)\equiv s^{\frac{p}{2}}$ and sub-quintic nonlinearity in a bounded smooth domain of $\mathbb{R}^3$ (see \cite{zhongyanzhu}). More recently, Zhou $et$ $al$ in \cite{zhou} investigated the dynamics of Eq. \eqref{1.1} with an additional weak anti-damping term $\mathcal{K}(\partial_{t}u)$ when the nonlinear term $g$ exhibits sub-quintic growth (see \cite{zhou}).

In this paper,  inspired by the aforementioned literature, we investigate the long-term dynamics of Eq. \eqref{1.1} with the nonlocal nonlinear damping term $\mathcal{J}(\|\partial_{t}u\|^{2})\partial_{t}u$ and quintic nonlinearity $g(u)$ satisfying the following hypotheses:
\begin{assumption}\label{A1.1}
\textbf{(J)}
$\mathcal{J}(\cdot)\in\mathcal{C}^{1}[0,+\infty)$ is strictly increasing, satisfying
\begin{enumerate}
\item either
\begin{align}\label{1.2}
\mathcal{J}(s)>0,\quad\forall s\in\mathbb{R}_{+};
\end{align}
\item or
\begin{align}\label{1.3}
s^{p+1}\leq\mathcal{J}(s)s,\quad\forall s\in\mathbb{R}_{+},
\end{align}
where $p$ is a given positive constant.
\end{enumerate}

\textbf{(GH)}
$g\in \mathcal{C}^{2}({\mathbb{R}})$ with $g(0)=0$ and
\begin{align}\label{1.4}
&|g''(s)|\leq C_{g}(1+|s|^{q-2}),\quad g'(s)\geq-\kappa_{1}+\kappa_{2}|s|^{q-1},
\\&  g(s)s-4G(s)\geq-\kappa_{3},\quad G(s)\geq \kappa_{4}|s|^{q+1}-\kappa_{5},~\forall s\in\mathbb{R},\label{1.5}
\end{align}
where $3\leq q\leq5$, $G(s)=\int_{0}^{s}g(\tau)d\tau$ and $\{\kappa_{i}\}_{i=1}^{5}$ are given positive constants. In addition, $h\in L^{2}(\Omega)$.
\end{assumption}

\begin{remark}
The presentation of the nonlocal damping coefficient $\mathcal{J}(\cdot)$ satisfying Assumption \ref{A1.1} \textbf{(J)} is based on general and abstract models. It covers not only the wave equations with linear damping ($\mathcal{J}(\cdot)\equiv const$), but also the wave equations with the damping coefficient is bounded when $\mathcal{J}(s)=\frac{a+s}{b+s}$ (hyperbolic function) or $\mathcal{J}(s)=\frac{ae^{s}}{1+be^{s}}$ (logistic function), where $a<b$ are positive constants. Another canonical example for $\mathcal{J}(s)$ is a power law, specifically $\mathcal{J}(s)=s^{p}$ or $\mathcal{J}(s)=s^{p}$+``lower order terms''.
\end{remark}

To investigate the dynamics of Eq. \eqref{1.1}, several intriguing questions emerge from the following aspects:
\begin{enumerate}
\item Dissipativity


The arbitrariness of the exponent $p$ for nonlocal damping, combined with the quintic nonlinearity $g$, introduces significant challenges in analyzing dissipativity. To address these challenges, we employed a new-type Gronwall's inequality constructed in \cite[Lemma 3.2]{zhou} to establish the dissipativity of the system generated by S--S solutions to problem \eqref{1.1}.

%
%
%
%
%

\item Asymptotic compactness

In the sub-quintic case, the so-called energy-to-Strichartz (ETS) estimate can be established as follows:
 \begin{align}\label{ETS1.5}
\|u\|_{L^4\left([t, t+1]; L^{12}(\Omega)\right)} \leqslant \mathcal{Q}\left(\left\|\xi_u(t)\right\|_{\mathscr{E}}\right)+\mathcal{Q}\left(\|h\|\right),
 \end{align}
where $\mathscr{E}=H_{0}^{1}(\Omega)\times L^{2}(\Omega)$ and $\mathcal{Q}$ is a monotone function independent of $u$ and $t$ (e.g., see \cite{chang,zhongyanzhu,zhou}). Utilizing the ETS estimate, one can derive asymptotic compactness and the existence of attractors in a manner similar to that used for the classical cubic or sub-cubic cases. In contrast, for the quintic case, the ETS estimate has only been established for $\Omega=\mathbb{R}^3$ or $\Omega=\mathbb{T}^3$ with periodic boundary conditions (e.g., see \cite{mssz} and the references therein).
To the best of our knowledge, the ETS estimate for general domains remains unresolved. Consequently, it is not possible to deduce asymptotic compactness through any control of the Strichartz norm in terms of the initial data.

To overcome the difficulties brought by the critical nonlinearity, several established techniques have been employed, including the so-called energy method developed by Ball \cite{ball}, the decomposition technique (e.g., see \cite{arrieta,zelik}) and the compensated compactness method (also known as the ``contractive function'' method, e.g., see \cite{khanmamedov1,lasiecka,sun}). These approaches have been effective in proving the asymptotic compactness of solutions. However, due to the quintic growth rate of the nonlinearity $g(u)$ in Eq. \eqref{1.1} and the unresolved status of the ETS estimate, these methods appear to be inapplicable to our context.

 %
%
%
%

In the quintic case, where $\mathcal{J}(\cdot)\equiv$ const $>0$, an intriguing approach to address the challenges posed by the quintic nonlinearity is presented by Zelik $et$ $al$ in \cite{ksz}:


\begin{enumerate}[(1)]
\item The existence and structure of the weak trajectory attractor $\mathscr{A}_{tr}=\Pi_{t\geq0}\mathscr{K}$ are established for the trajectory dynamical system generated by the Galerkin solution of problem \eqref{1.1}. Here, $\Pi_{t\geq0}u:=u|_{t\geq0}$ denotes the restriction of $u$ to $t \geq 0$, and $\mathscr{K}$ represents the set of all complete solutions to equation \eqref{1.1}.

%

\item It is shown that, every complete solution $u(t)$, $t\in\mathbb{R}$, within the weak trajectory attractor $\mathscr{A}_{tr}$ is a global strong regular solution of Eq. \eqref{1.1}. Specifically, these solutions satisfy the energy identity.

\item The existence and regularity of a compact global attractor are achieved using an energy method combined with a decomposition technique.

\end{enumerate}
As highlighted by the authors in \cite{ksz}, the trajectory attractor for the Galerkin solutions and the backward regularity of the complete solution $u(t)$ within the weak trajectory attractor $\mathscr{A}_{tr}$ play a pivotal role in establishing asymptotic compactness.
However, in our context, the presence of the nonlinear nonlocal damping term $\mathcal{J}(\|\partial_{t}u(t)\|^{2})\partial_{t}u$ complicates the application of this approach, presenting significant challenges:

\begin{itemize}
\item 

To our knowledge, the existence of the Galerkin solution for Eq. \eqref{1.1} remains unclear. The primary challenge lies in estimating energy boundedness, where we can only establish the boundedness of the Galerkin approximation $\partial_t u_N$ in the $L^2(\Omega)$ norm. This yields weak convergence $\partial_t u_N \rightharpoonup \partial_t u$ in the $L^2(\Omega)$ norm but does not guarantee that the nonlocal coefficients $\mathcal{J}\left(\left\|\partial_t u_N(t)\right\|^2\right)$ converge to $\mathcal{J}\left(\left\|\partial_t u(t)\right\|^2\right)$.


    To address this difficulty, several authors have explored the monotonicity method; see \cite{chueshov2, lions, zhong, zhou} for further details. However, this method is not applicable to our case. Specifically, to ensure that the trajectory phase space $\mathscr{K}^{+}$generated by Galerkin solutions is closed with respect to the topology induced by the embedding $\mathscr{K} \subset \Theta^{+}:=\left[L_{l o c}^{\infty}\left(\mathbb{R}_{+}, \mathscr{E}\right)\right]^{w^*}$, the Galerkin solution must be defined as a weak-star limit in $L^{\infty}(0, T ; \mathscr{E})$ of the Galerkin approximation $u_N$. Moreover, the strong convergence of the initial data in the energy space $\mathscr{E}$ cannot be assumed, which is crucial for the monotone operator method.

\item 
    In the case of linear damping, let $u_n \in \mathcal{C}([0, \infty) ; \mathscr{E})$ be a sequence of general weak (or S--S ) solutions to Eq. \eqref{1.1}. Under appropriate conditions, one can extract a subsequence $u_{n_k}$ from $u_n$ that converges to some $u \in \mathcal{C}\left([0, \infty) ; \mathscr{E}_w\right)$ as $n_k \rightarrow \infty$. Moreover, this limit $u$ remains a weak solution of Eq. \eqref{1.1}. In the scenario with nonlinear damping, it is not known whether the limit $u$ retains the property of being a weak solution to Eq. \eqref{1.1}. Consequently, the attractor might include solutions that are less regular than the S--S solutions or even functions that are not solutions to Eq. \eqref{1.1}. Thus, it becomes challenging to apply standard methods to establish the existence and structure of the trajectory attractor $\mathscr{A}_{t r}$ for the trajectory space generated by the general weak solutions of Eq. \eqref{1.1}.
\end{itemize}


Thus, a relevant question arises: is it possible to achieve asymptotic compactness for the dynamical system associated with Eq. \eqref{1.1}, particularly when it involves a general nonlinear nonlocal damping term? If this is not feasible, are there alternative methods that could provide insight into the dynamics? Exploring these questions will require the development of new methods and theories.

\item Smoothness and finite-dimensionality

In the sub-quintic case, using the standard bootstrapping arguments one can easily show that the attractor $\mathscr{A}$ for the wave Eq. \eqref{1.1} is in a more regular energy space $\mathscr{E}^{1}=\mathcal{H}^{2}\times\mathcal{H}^{1}$ ($\mathcal{H}^{s}=D((-\Delta)^{\frac{s}{2}}$), see \cite{chang,ksz} for more details. In the quintic case, when $\Omega$ is a bounded domain and $\mathcal{J}(\cdot)\equiv$ const $>0$, the smoothness of the attractor has been explored in \cite{ksz,zhongyt}. For the non-autonomous case, further regularity of the attractor has been established only for $\Omega=\mathbb{R}^3$ or $\Omega=\mathbb{T}^3$ with periodic boundary conditions, as discussed in \cite{mssz,sav2}. Up to now and to the best of our knowledge, the study of the smoothness and fractal dimension of attractor for the wave equation \eqref{1.1} defined on bounded domains with both quintic nonlinearity and nonlocal nonlinear damping term is still lacking.



We now address the following fundamental question: Is it possible to establish the existence of a finite-dimensional attractor for Equation $\eqref{1.1}$ in a higher regular phase space, such as $\mathcal{H}^2 \times \mathcal{H}^1$ ? To the best of our knowledge, no results currently address this issue, and novel approaches are required to provide an answer.
\end{enumerate}


In this paper, to confront the aforementioned challenges, we propose a novel approach for analyzing the dynamics of the wave equation \eqref{1.1}. The primary method is illustrated in Figure \ref{f1} for clarity.
\begin{enumerate}[I.]
\item In the quintic case, by applying the Gronwall's inequality established in \cite{zhou}, we obtain the dissipative of the dynamical systems in $\mathscr{E}$.
\item We utilize a newly developed framework called evolutionary systems (see \cite{cheskidov}) to study the asymptotic dynamics of S--S solutions, and thereby establish the existence and structure of the weak global attractor $\mathscr{A}_{w}$. Since the evolutionary systems $\mathfrak{E}$ generated by S--S solutions may not closed with respect to weak topology on the phase space $\mathscr{E}$, we adopt an insightful technique introduced by Cheskidov and Lu in \cite{cheskidov5}, which involves taking the closure of the evolutionary systems $\bar{\mathfrak{E}}$. Our main objective is to demonstrate that $\mathfrak{E}((-\infty, \infty))=$ $\overline{\mathfrak{E}}((-\infty, \infty))$ using a newly developed method outlined in \cite{zelik1}. By exploiting the backward regularity of complete trajectories within $\mathfrak{E}((-\infty, \infty))$ along with the standard energy method, we establish the asymptotic compactness of the S--S solutions and ultimately prove that the weak global attractor $\mathscr{A}_w$ is indeed a strongly compact attractor $\mathscr{A}_s$.

\item We investigate the strong attractor for S--S solution semigroups when restricted in $\mathscr{E}^{1}$. Taking advantage of the fact that the dynamical system $(S(t), \mathscr{E})$ has a compact global attractor, we establish the dissipativity of $\left(S(t), \mathscr{E}^1\right)$ using a decomposition technique. Subsequently, we establish the existence of the exponential attractor $\mathfrak{A}=\{\mathscr{A}_{exp}(s):s\in\mathbb{R}\}$ through a quasi-satble method. Utilizing the known results that $\mathscr{A}_{s}\subset\mathscr{A}_{exp}(s)$, $\forall s\in\mathbb{R}$ and applying attraction transitivity result, we ultimately prove that the global attractor $\mathscr{A}_{s}$ is compact in $\mathscr{E}^{1}$ and that its fractal dimension is finite.

%

\end{enumerate}
\begin{figure}[t]
\centering
\tikzstyle{format}=[rectangle,draw,thin,fill=white]
\tikzstyle{test}=[diamond,aspect=2,draw,thin,font=tiny]
\tikzstyle{point}=[coordinate,on grid,]
{\scalefont{0.7}
\begin{tikzpicture}
		\node[format,align=center] (SSS){Shatah--Struwe solutions};
		\node[format,below of=SSS,node distance=0.7cm,align=center](ESE){Evolutionary systems $\mathfrak{E}$};
        \node[format,below of=ESE,node distance=1cm,align=center](ESE1){Evolutionary systems \\ $\bar{\mathfrak{E}}$};
        \node[format,below of=ESE1,node distance=1.4cm,align=center](AW){Weak attractor \\$\mathscr{A}_{w}$};
        \node[format,left of=ESE1,node distance=2.8cm,align=center](DE){Dissipative in \\$\mathscr{E}$};
        \node[format,left of=AW,node distance=2.3cm,align=center](A1l){$A_{1}$--Property};

        \node[format,below of=AW,node distance=1.4cm,align=center](AW1)
        {$\mathscr{A}_{w}=\{\xi_{u}(0)\}$\\$\xi_{u}\in\bar{\mathfrak{E}}((-\infty,\infty))$};
       \node[format,below of=AW1,node distance=1.7cm,align=center](SA){$\mathscr{A}_{w}=\mathscr{A}_{s}$};
       \node[format,below of=SA,node distance=1.5cm,align=center](AB1){$\mathscr{A}_{s}$ bounded \\ in $\mathscr{E}^{1}$};

        \node[format,right of=AW1,node distance=3cm,align=center](EE){$\mathfrak{E}((-\infty,\infty))$\\=$
       \bar{\mathfrak{E}}((-\infty,\infty))$};

        \node[format,above of=EE,node distance=1.5cm,align=center](BAR){Backward asymptotic \\ regularity};

       \node[format,below of=EE,node distance=1.1cm,align=center](AC){Asymptotic compact};

        \node[format,right of=A1l,node distance=9.2cm,align=center](SS){Strong solutions};
        \node[format,below of=SS,node distance=1cm,align=center](SE1){$(S(t),\mathscr{E}^{1})$};
         \node[format,below of=SE1,node distance=2.1cm,align=center](DE1){Dissipative in \\$\mathscr{E}^{1}$};
         \node[format,right of=DE1,node distance=2.2cm,align=center](QM){Quasi-stable\\ method};
        \node[format,below of=AC,node distance=3.5cm,align=center](Final){$\mathscr{A}_{s}\Subset\mathscr{E}^{1}$,
        $\dim_{\mathscr{F}}^{\mathscr{E}^{1}}(\mathscr{A}_{s})<\infty$};
        \node[format,right of=AB1,node distance=6.9cm,align=center](EA){Exponential attractor \\$\mathfrak{A}$};
        \node[format,right of=EE,node distance=2cm,align=center](EM){Energy\\ method};
         \node[format,below of=AC,node distance=1cm,align=center](DM){Decomposition};

        \node[point,above of=EE,node distance=8mm](point1){};
        \node[point,below of=ESE1,node distance=7mm](point2){};
        \node[point,below of=AW,node distance=7mm](point3){};
        \node[point,below of=EE,node distance=7mm](point4){};
        \node[point,below of=DE1,node distance=7mm](point5){};
        \node[point,below of=SE1,node distance=7mm](point6){};
        \node[point,above of=DE1,node distance=10mm](point7){};
        \node[point,right of=DM,node distance=2.5cm](point8){};
        \node[point,below of=BAR,node distance=7mm](point9){};
         \node[point,right of=AW1,node distance=16mm](point10){};

\draw[->](SSS)--(ESE);
\draw[->](ESE)--(ESE1);
\draw[->](ESE1)--(AW);
\draw[->](AW)--(AW1);
\draw[->](AW1)--(SA);
\draw[->](SA)--(AB1);
\draw[->](SS)--(SE1);
\draw[->](SE1)--(DE1);
\draw[->](DE1)--(EA);
\draw[->](EE)--(AC);
\draw[-](AC)-|(SA);
\draw[-](DM)-|(AB1);
\draw[-](AB1)--(EA);
\draw[->](AC)--(DM);
\draw[->](BAR)--(EE);
\draw[->](EA)-|(Final);

\draw[-](DE)|-(point2);
\draw[-](A1l)|-(point3);
\draw[-](EM)|-(point4);
\draw[-](QM)|-(point5);
\draw[-](DM)--(point8);
\draw[-](point7)-|(point8);
\draw[-](AW1)|-(point10);
\draw[-](point10)|-(point9);
\end{tikzpicture}
\begin{center}
\raggedright
1. $\mathscr{E}^{s}=\mathcal{H}^{s+1}\times\mathcal{H}^{s}$, $\mathcal{H}^{s}=D((-\Delta)^{\frac{s}{2}})$, $s\in\mathbb{R}$. 2. $\bar{\mathfrak{E}}$: the closure of $\mathfrak{E}$ in the topology generated by $\mathcal{C}([a,\infty);\mathscr{E}_{w})$. 3. $\dim_{\mathscr{F}}^{\mathscr{E}^{1}}(\mathscr{A}_{s})$: the fractal dimension of $\mathscr{A}_{s}$ in space $\mathscr{E}^{1}$. 4. $\mathscr{A}_{s}\Subset\mathscr{E}^{1}$: the embedding $\mathscr{A}_{s}\subset\mathscr{E}^{1}$ is compact. 5. $A_{1}$--Property: $\mathfrak{E}([0,\infty))$ is pre-compact in $\mathcal{C}([0,\infty); \mathscr{E}_{w})$.
\end{center}
}
\caption{Overview of the technique.}\label{f1}
\end{figure}

The structure of our paper is outlined as follows. In Section 2, we provide a brief overview of the theory of evolutionary systems. In Section 3, the global existence and dissipativity of the S--S solutions of Eq. \eqref{1.1} are discussed in Theorem \ref{ZFT1} and Theorem \ref{ZFT2}. Then the existence, structure and properties of the weak global attractor are studied in Theorem \ref{ZFT3} and Theorem \ref{ZFT4}. In Section 4, the backward asymptotic regularity of complete trajectories within $\bar{\mathfrak{E}}((-\infty,\infty))$ is proved in Theorem \ref{ZFT5}. Following this, Theorem \ref{ZFT7} demonstrates the existence of the strong global attractor $\mathscr{A}_s$. In Section 6, we prove the existence of the exponential attractor for the strong solution of problem \eqref{1.1} in Theorem \ref{ZFT9}. Finally, Theorem \ref{ZFT10} establishes the higher regularity and finite fractal dimension of the global attractor $\mathscr{A}_s$.

Throughout the paper, $\mathcal{Q}(\cdot)$ denotes a monotone increasing function, while $C$ represents a generic constant, with indices used for clarity as needed. Additionally, distinct positive constants $C_i$, where $i \in \mathbb{N}$, are employed for specific differentiation purposes throughout the discussion.
\section{Preliminaries}
Let $\|\cdot\|$ and $\langle\cdot,\cdot\rangle$ be the usual norm and inner product in $L^{2}(\Omega)$.
For convenience, we denote $L^p=L^p(\Omega)$ ($p\geq1$), $H_0^1=H_0^1(\Omega)$, $H^2=H^2(\Omega)$. Let $\mathcal{H}^{s}=D((-\Delta)^{\frac{s}{2}})$, $\mathscr{E}^{s}=\mathcal{H}^{s+1}\times\mathcal{H}^{s}$, $s\in\mathbb{R}.$ Then, $\mathcal{H}^{0}=L^{2}$, $\mathcal{H}^{1}=H_{0}^{1}$, $\mathcal{H}^{2}=H^{2}\cap H_{0}^{1}$, and $\mathcal{H}^{-1}$ is the dual space to $H_{0}^{1}$. In particular, we denote $\mathscr{E}:=\mathscr{E}^{0}=H_{0}^{1}\times L^{2}$ and denote $\langle\langle\cdot,\cdot\rangle\rangle$ the inner product in $\mathcal{H}^{1}$.

\subsection{Strichartz estimates}
Consider the linear wave equation
\begin{align}\label{2.3}
\begin{cases}
\partial_{t}^{2}u-\Delta u=h(t), \quad\text{in~}\Omega\times\mathbb{R},
\\u(x,0)=u^{0}, ~\partial_{t}u(x,0)=u^{1}.
\end{cases}
\end{align}
Then we have the following so-called Strichartz estimates, and its proof can be found in \cite{bss}.
\begin{lemma}\label{SL}
Suppose $2<p_{1}\leq\infty$, $2\leq q_{1}<\infty$ and $(p_{1},q_{1},r_{1})$ is a triple satisfying
\begin{align}\label{2.4}
\frac{1}{p_{1}}+\frac{3}{q_{1}}=\frac{3}{2}-r_{1},\quad
\frac{3}{p_{1}}+\frac{2}{q_{1}}\leq1,
\end{align}
and $(p_{2}',q_{2}',1-r_{1})$ also satisfies the above conditions. Then we have the following estimates for solutions $u$ to \eqref{2.3} satisfying Dirichlet or Neumann homogeneous boundary conditions
\begin{align}
\|u\|_{L^{p_{1}}([-T,T];L^{q_{1}})}
\leq C\Big(\|u^{0}\|_{\mathcal{H}^{r_{1}}}+\|u^{1}\|_{\mathcal{H}^{r_{1}-1}}
+\|h\|_{L^{p_{2}}([-T,T];L^{q_{2}})}\Big)
\end{align}
with $C$ some positive constant may depending on $T$.
\end{lemma}

Indeed, when $p_{1}=4$, $q_{1}=12$, $r_{1}=1$, $p_{2}=1$ and $q_{2}=2$, we get the important special case
\begin{align}\label{2.6*}
\|u\|_{L^{4}([-T,T];L^{12})}
\leq C\Big(\|u^{0}\|_{\mathcal{H}^{1}}+\|u^{1}\|_{L^{2}}
+\|h\|_{L^{1}([-T,T];L^{2})}\Big).
\end{align}
\subsection{Evolutionary systems}
We recall some basic ideas and results from the abstract theory of evolutionary systems, see \cite{cheskidov,cheskidov3,cheskidov4,cheskidov5} for details. Let $(\mathcal{X},d_{s}(\cdot,\cdot))$ be a metric space endowed with a metric $d_{s}$, which will be referred to as a strong metric. Let $d_{w}(\cdot,\cdot)$ be another metric on $\mathcal{X}$ satisfying the following conditions:
\begin{enumerate}
\item $\mathcal{X}$ is $d_{w}$-compact.
\item If $d_{s}(u_{n},v_{n})\rightarrow0$ as $n\rightarrow\infty$ for some $u_{n}$, $v_{n}\in\mathcal{X}$, then $d_{w}(u_{n},v_{n})\rightarrow0$.
\end{enumerate}
Due to property 2, $d_{w}(\cdot,\cdot)$ and $d_{s}(\cdot,\cdot)$ will be referred to as weak metric and strong metric respectively. Let $\mathcal{C}([a,b];\mathcal{X}_{\bullet})$, where $\bullet=s$ or $w$, be the space of $d_{\bullet}$-continuous $\mathcal{X}$-valued functions on $[s,t]$ endowed with the metric
\begin{align*}
d_{\mathcal{C}([a,b];\mathcal{X}_{\bullet})}(u,v):=\sup_{t\in[a,b]}d_{\bullet}(u(t),v(t)).
\end{align*}
Let also $\mathcal{C}([a,\infty);\mathcal{X}_{\bullet})$ be the space of $d_{\bullet}$-continuous $\mathcal{X}$-valued functions on $[a,\infty)$ endowed with the metric
\begin{align}\label{2.5}
d_{\mathcal{C}([a,\infty);\mathcal{X}_{\bullet})}(u,v):=\sum_{K\in\mathbb{N}}
\frac{1}{2^{K}}\frac{d_{\mathcal{C}([a,a+K];\mathcal{X}_{\bullet})}(u,v)}
{1+d_{\mathcal{C}([a,a+K];\mathcal{X}_{\bullet})}(u,v)}.
\end{align}
To define an evolutionary systems, first let
\begin{align*}
\mathcal{T}:=\{I:I=[T,\infty)\subset\mathbb{R},~\text{or }I=(-\infty,\infty)\},
\end{align*}
and for each $I\in\mathcal{T}$, let $\mathfrak{F}(I)$ denote the set of all $\mathcal{X}$-valued functions on $I$.
\begin{definition}\label{Devolution}
A map $\mathfrak{E}$ that associates to each $I\in\mathcal{T}$ a subset $\mathfrak{E}(I)\subset\mathfrak{F}(I)$ will be called an evolutionary system if the following conditions are satisfied:
\begin{enumerate}[1.]
\item $\mathfrak{E}([0,\infty))\neq\emptyset$.
\item $\mathfrak{E}(I+s)=\{u(\cdot):u(\cdot-s)\in\mathfrak{E}(I)\}$ for all $s\in\mathbb{R}$.
\item $\{u(\cdot)\mid_{I_{2}}:u(\cdot)\in\mathfrak{E}(I_{1})\}\subset\mathfrak{E}(I_{2})$ for all pairs $I_{1}$, $I_{2}\subset\mathcal{T}$, such that $I_{2}\subset I_{1}$.
\item $\mathfrak{E}((-\infty,\infty))=\{u(\cdot):u(\cdot)\mid_{[T,\infty)}
    \in\mathfrak{E}([T,\infty)),\forall T\in\mathbb{R}\}$.
\end{enumerate}
\end{definition}
We will refer to $\mathfrak{E}(I)$ as the set of all trajectories on the time interval $I$. Let $P(\mathcal{X})$ be the set of all subsets of $\mathcal{X}$. For every $t \geq 0$, define a map
$$
\begin{gathered}
R(t): P(\mathcal{X}) \rightarrow P(\mathcal{X}), \\
R(t) A:=\{u(t): u(0) \in A, u \in \mathfrak{E}([0, \infty))\}, \quad A \subset\mathcal{X}.
\end{gathered}
$$
\begin{definition}
A set $\mathscr{A}_{w} \subset \mathcal{X}$ is a d$_{w}$-global attractor of $\mathfrak{E}$ if $\mathscr{A}_{w}$ is a minimal set that is
\begin{enumerate}[1.]
\item  d$_{w}$-closed;
\item d$_{w}$-attracting: for any $B \subset\mathcal{X}$ and $\epsilon>0$, there exists $t_0$, such that
$$
R(t) B \subset B_{w}\left(\mathscr{A}_{w}, \epsilon\right):=\left\{u: \inf _{x \in \mathscr{A}_{w}} d_{w}(u, x)<\epsilon\right\}, \quad \forall t \geq t_0.
$$
\end{enumerate}
\end{definition}
\begin{definition}
The $\omega_{\bullet}$-limit set $(\bullet=s, w)$ of a set $A \subset \mathcal{X}$ is
$$
\omega_{\bullet}(A):=\bigcap_{T \geq 0} \overline{\bigcup_{t \geq T} R(t) A}^{\bullet}.
$$
\end{definition}
In order to extend the notion of invariance from a semigroup to an evolutionary system, we will need the following mapping:
$$
\widetilde{R}(t) A:=\{u(t): u(0) \in A, u \in \mathfrak{E}((-\infty, \infty))\}, \quad A \subset\mathcal{X},~t \in \mathbb{R} .
$$

\begin{definition}
A set $A \subset\mathcal{X}$ is positively invariant if
$$
\widetilde{R}(t) A \subset A, \quad \forall t \geq 0 .
$$
$A$ is invariant if
$$
\widetilde{R}(t) A=A, \quad \forall t \geq 0 .
$$
$A$ is quasi-invariant, if for every $a \in A$, there exists a complete trajectory $u \in$ $\mathfrak{E}((-\infty, \infty))$ with $u(0)=a$ and $u(t) \in A$ for all $t \in \mathbb{R}$.
\end{definition}
As shown in \cite{cheskidov,cheskidov5}, a semigroup $\{S(t)\}_{t\geq0}$ defines an evolutionary system. In order to investigate the existence and structure of $\mathscr{A}_{w}$, we use a new method initiated by Cheskidov and Lu in \cite{cheskidov5} by taking a closure of the evolutionary system $\mathfrak{E}$. Let
\begin{align*}
\bar{\mathfrak{E}}([\tau,\infty)):=\overline{\mathfrak{E}([\tau,\infty))}^{\mathcal{C}([\tau,\infty);\mathcal{X}_{w})},\quad\forall\tau\in\mathbb{R}.
\end{align*}
Obviously, $\bar{\mathfrak{E}}$ is also an evolutionary system. We call $\bar{\mathfrak{E}}$ the closure of the evolutionary system $\mathfrak{E}$, and add the top-script $^{-}$ to the corresponding notations. Below is an important property for $\mathfrak{E}$ in some cases.
\begin{itemize}
\item[$\diamondsuit$] \textbf{A1} $\mathfrak{E}([0,\infty))$ is pre-compact in $\mathcal{C}([0,\infty); \mathcal{X}_{w})$.
\end{itemize}
\begin{theorem}(\cite{cheskidov5})\label{ZFT-ES}
Assume $\mathfrak{E}$ is an evolutionary system. Then the weak global attractor $\mathscr{A}_{w}$ exists. Furthermore, assume that $\mathfrak{E}$ satisfies \textbf{A1}. Let $\bar{\mathfrak{E}}$ be the closure of $\mathfrak{E}$. Then
\begin{enumerate}[1.]
\item $\mathscr{A}_{w}=\omega_{w}(\mathcal{X})=\bar{\omega}_{w}(\mathcal{X})=\bar{\mathscr{A}}_{w}
=\{u_{0}\in\mathcal{X}:u_{0}=u(0)\text{~for~some~}u\in\bar{\mathfrak{E}}((-\infty,\infty))\}$.
\item $\mathscr{A}_{w}$ is the maximal invariant and maximal quasi-invariant set w.r.t. $\bar{\mathfrak{E}}$.
\item (Weak uniform tracking property) For any $\epsilon>0$, there exists $t_{0}$, such that for any $t^{*}>t_{0}$, every trajectory $u\in\mathfrak{E}([0,\infty))$ satisfies
    \begin{align*}
    d_{\mathcal{C}([t^{*},\infty);\mathcal{X}_{w})}(u,v)\leq\epsilon,
    \end{align*}
for some complete trajectory $v\in\bar{\mathfrak{E}}((-\infty,\infty))$.
\end{enumerate}
\end{theorem}
\section{Weak attractors}
\subsection{Well-posedness and dissipativity}
\begin{definition}\label{ZFD}
A function $u(t)$ is a
\begin{itemize}
\item (W) weak solution of Eq. \eqref{1.1} iff $\xi_{u}(t):=(u(t),\partial_{t}u(t))\in L^{\infty}(0,T;\mathscr{E})$ and Eq. \eqref{1.1} is satisfied in the sense of distribution, i.e.
    \begin{align*}
    \nonumber-&\int_{0}^{T}\langle\partial_{t}u,\partial_{t}\phi\rangle dt+
    \mathcal{J}(\|\partial_{t}u(t)\|^{2})\int_{0}^{T}\langle\partial_{t}u,\phi\rangle dt+\int_{0}^{T}\langle\nabla u\cdot\nabla\phi,1\rangle dt
    \\&+\int_{0}^{T}\langle g(u),\phi\rangle dt
    =\int_{0}^{T}\langle h,\phi\rangle dt
    \end{align*}
    for any $\phi\in\mathcal{C}_{0}^{\infty}((0,T)\times\Omega)$.
\item (S--S) Shatah--Struwe solution of Eq. \eqref{1.1} on the interval $[0,T]$ iff $u(t)$ is a weak solution and
    \begin{align*}
    u\in L^{4}([0,T];L^{12}).
    \end{align*}
\item (S) strong solution of Eq. \eqref{1.1} on the interval $[0,T]$ iff
\begin{enumerate}[(i)]
\item $u\in W^{1,1}(r,s;\mathcal{H}^{1})$ and $\partial_{t}u\in W^{1,1}(r,s;L^{2})$ for any $0<r<s<T$;
\item $-\Delta u(t)+\mathcal{J}(\|\partial_{t}u(t)\|^{2})\partial_{t}u\in L^{2}$ for a.e. $t\in[0,T]$;
\item Eq. \eqref{1.1} is satisfied in $L^{2}$ for a.e. $t\in[0,T]$.
\end{enumerate}
\end{itemize}
\end{definition}
\begin{theorem}\label{ZFT1}(\cite{zhou})
Let $\mathcal{J}(\cdot)$, $g$ and $h$ satisfy Assumption \ref{A1.1}. For any initial condition $\xi_{u}(0)\in\mathscr{E}$, there exists a unique global S--S solution $u(t)$ of Eq. \eqref{1.1} satisfying the energy equality
\begin{align}\label{Zfeng3.1}
\mathcal{E}_{u}(T)+2\int_{0}^{T}\mathcal{J}(\|\partial_{t}u(t)\|^{2})\|\partial_{t}u(t)\|^{2}dt=\mathcal{E}_{u}(0),\quad \forall T\geq 0
\end{align}
and the following Strichartz estimate
\begin{align}\label{Z3.1}
\|u\|_{L^{4}(0,T;L^{12})}
\leq\mathscr{Q}_{T}(\xi_{u}(0),\|h\|^{2}),
\end{align}
where $\mathcal{E}_{u}(t)=\|\xi_{u}\|_{\mathscr{E}}^{2}+2\langle G(u),1\rangle-2\langle h,u\rangle$ and the function $\mathscr{Q}_{T}$ is increasing in $T$.  Furthermore, if $\xi_{u}(0)\in\mathscr{E}^{1}$, then the corresponding S--S solution is the strong solution of Eq. \eqref{1.1}.
\end{theorem}
From Theorem \ref{ZFT1}, we can define the operators
$S(t): \mathscr{E}\rightarrow\mathscr{E}$ by
\begin{align}\label{Z3.2}
S(t)\xi_{u}(0):=\xi_{u}(t)
,\quad\xi_{u}(0)\in\mathscr{E},\quad t\geq0,
\end{align}
where $\xi_{u}(t)$ is the unique S--S solution to Eq. \eqref{1.1}. Obviously, we can conclude that the family of operators
$\{S(t)\}_{t\geq0}$ defined by \eqref{Z3.2} is a semigroup. Moreover, we have
\begin{align}\label{Z3.3}
\|\xi_{u}(t)-\xi_{v}(t)\|_{\mathscr{E}}\leq
L\|\xi_{u}(0)-\xi_{v}(0)\|_{\mathscr{E}},
\end{align}
where $L=e^{C\int_{0}^{t}(1+\|u\|_{L^{12}}^{4}+\|v\|_{L^{12}}^{4})dr}$, and positive constant $C$ depends only on the coefficients in \eqref{1.4}.

\begin{theorem}\label{ZFT2}
Under Assumption \ref{A1.1} the semigroup $S(t)$ defined by \eqref{Z3.3} is ultimately dissipative. More precisely, there exists $R_{0}>0$ possessing the following property: for any bounded subset $B\subset\mathscr{E}$, there is a $T=T(B)$ such that
\begin{align}\label{Z3.4}
\|S(t)\xi_{u}(0)\|_{\mathscr{E}}\leq R_{0}
\end{align}
for all $\xi_{u}(0)\in B$ and $t\geq T$, where $R_{0}$ may be depend on $\|h\|$, $|\Omega|$, $\lambda_{1}$ (the first eigenvalue of $-\Delta$ in $\mathcal{H}^{1}$) and the other structural parameters of Eq. \eqref{1.1} appearing in Assumption \ref{A1.1}.
\end{theorem}
\begin{proof}
The proof of this fact follows the same arguments as those in \cite[Theorem 3.3]{zhou}, and therefore it is omitted here.
\end{proof}

\subsection{Properties of weak attractors}
According to Theorem \ref{ZFT2}, we may, without loss of generality, assume that the bounded absorbing set
\begin{align}\label{ZHOU3.5}
\mathds{B}_{0}:=\{\xi_{u}\in\mathscr{E}:\|\xi_{u}\|_{\mathscr{E}}\leq R_{0}\}
\end{align}
is positively invariant with respect to the S--S solution semigroup $\{S(t)\}_{t\geq0}$. We now define an evolutionary system (ES) on $\mathscr{E}$ by
\begin{align}\label{Z3.5}
\mathfrak{E}([0,\infty)):=\{\xi_{u}(\cdot):\xi_{u}(t)=S(t)\xi_{u}(0),
~\xi_{u}(t)\in\mathcal{X},~\forall t\geq0\},
\end{align}
where $\mathcal{X}:=\{\xi_{u}\in\mathscr{E}:\|\xi_{u}\|_{\mathscr{E}}\leq R_{0}\}$. Let
\begin{align}\label{Z3.6}
\bar{\mathfrak{E}}([0,\infty)):=\overline{\mathfrak{E}([0,\infty))}^{\mathcal{C}([0,\infty);\mathcal{X}_{w})},
\end{align}
where the metric on $\mathcal{C}([0,\infty);\mathcal{X}_{\bullet})$ is defined similarly to that in \eqref{2.5}.

\begin{lemma}\label{ZF-LW}
Assuming the conditions of Theorem \ref{ZFT2} are met, and given that $\xi_{u_{n}}=(u_{n},\partial_{t}u_{n})$ represents a sequence of S--S solutions to Eq. \eqref{1.1} with $\xi_{u_{n}}(t)\in\mathcal{X}$ for all $t\geq t_{0}$. Then
\begin{align}\label{Z3.7}
\xi_{u_{n}}\text{ is bounded in } L^{\infty}([t_{0},T];\mathscr{E}),\quad \partial_{t}\xi_{u_{n}}\text{ is bounded in } L^{\infty}([t_{0},T];\mathscr{E}^{-1}),~~\forall T>t_{0}.
\end{align}
Furthermore, there exists a subsequence $n_j$ such that $\xi_{u_{n_j}}$ converges to some $\xi_u$ in $\mathcal{C}\left(\left[t_0, T\right] ; \mathscr{E}_w\right)$, meaning that $(\xi_{u_{n_{j}}},\phi)\rightarrow(\xi_{u},\phi)$
uniformly on $[t_{0},T]$ as $n_{j}\rightarrow\infty$ for all $\phi\in\mathscr{E}$.
\end{lemma}
\begin{proof}
Applying Theorem \ref{ZFT2} and noting that $\xi_{u_{n}}$ are S--S solutions of Eq. \eqref{1.1}, we express the second derivative $\partial_t^2 u(t)$  as indicated in Eq. \eqref{1.1}, leading to \eqref{Z3.7}. By invoking Alaoglu's compactness theorem, we extract a subsequence $\xi_{u_{n_{j}}}$ which weak$^{\ast}$--converges to some function $\xi_{u}\in L^{\infty}([t_{0},T];\mathscr{E})$, i.e.,
\begin{align}\label{Z3.8}
\xi_{u_{n_{j}}}\rightharpoonup\xi_{u} \text{ weakly--$\ast$ in }L^{\infty}([t_{0},T];\mathscr{E}).
\end{align}
Utilizing the compact embedding result:
\begin{align*}
&\{(u,\partial_{t}u)\in L^{\infty}([t_{0},T];\mathscr{E})\}\cap\{\partial_{t}^{2}u\in L^{\infty}([t_{0},T];\mathcal{H}^{-1})\}
\\&\Subset\{(u,\partial_{t}u)\in\mathcal{C}([t_{0},T];\mathscr{E}^{-\varsigma})\}
\end{align*}
for some $0<\varsigma\leq 1$, we deduce that the weak--$\ast$ convergence \eqref{Z3.8} implies the strong convergence
$\xi_{u_{n_{j}}}\rightarrow\xi_{u} \text{~in~} \mathcal{C}([t_{0},T];\mathscr{E}_{w})$. The proof is complete.
\end{proof}

\begin{theorem}\label{ZFT3}
Let Assumption \ref{A1.1} be in force. Then the weak global attractor $\mathscr{A}_{w}$ for ES $\mathfrak{E}$ as defined in \eqref{Z3.5} exists.
Furthermore, $\mathfrak{E}$ satisfies condition \textbf{A1}, and the weak global attractor is given by
$$
\mathscr{A}_{w}:=\{\xi_{u_{0}}:\xi_{u_{0}}=\xi_{u}(0)\text{~for~some~}\xi_{u}\in\bar{\mathfrak{E}}((-\infty,\infty))\}.
$$
Additionally, for every $\epsilon>0$, there exists a time $t_{0}:=t_{0}(\epsilon)$ such that for any $t^{*}>t_{0}$, every trajectory $\xi_{v}\in\mathfrak{E}([0,+\infty))$ satisfies $d_{\mathcal{C}([0,\infty):\mathcal{X}_{w})}(\xi_{v},\xi_{u})<\epsilon$ for some complete trajectory $\xi_{u}\in\bar{\mathfrak{E}}((-\infty,\infty))$.
\end{theorem}

\begin{proof}
The existence of the attractor $\mathscr{A}_{w}$ can be established by using Theorem \ref{ZFT-ES} directly. Let $\xi_{u_{n}}$ be a sequence in $\mathfrak{E}([0,\infty))$. Using Lemma \ref{ZF-LW}, we can extract a subsequence (still denoted by $\xi_{u_{n}}$) that converges to some $\xi_{u}^{(1)}\in\mathcal{C}([0,1];\mathcal{X}_{w})$ as $n\rightarrow\infty$. Passing to a subsequence and still denoting $\xi_{u_{n}}$ once more, we obtain that $\xi_{u_{n}}\rightarrow\xi_{u}^{(2)}\in\mathcal{C}([0,2];\mathcal{X}_{w})$ as $n\rightarrow\infty$ for some $\xi_{u}^{(2)}\in\mathcal{C}([0,2];\mathcal{X}_{w})$ with $\xi_{u}^{(1)}=\xi_{u}^{(2)}$ on $[0,1]$. Continuing this diagonalization process, we get a subsequence $\xi_{u_{n_{j}}}$ converges to $\xi_{u}\in\mathcal{C}([0,\infty);\mathcal{X}_{w})$, and \textbf{A1} is proven. The other statements contained in the above theorem can be proved by applying Theorem \ref{ZFT-ES} again.
\end{proof}

\begin{theorem}\label{ZFT4}
Under the Assumption \ref{A1.1}, the complete trajectory $\xi_{u}\in\bar{\mathfrak{E}}((-\infty,\infty))$ if and only if there exists a sequence of times $t_{n}\rightarrow-\infty$ and a sequence of S--S solutions $\xi_{u_{n}}(t)$ of Eq. \eqref{1.1} given by:
\begin{equation}\label{Z3.9}
\begin{cases}
\partial_{t}^{2}u_{n}-\Delta u_{n}+\mathcal{J}(\|\partial_{t}u_{n}\|^{2})\partial_{t}u_{n}
+g(u_{n})=h(x),
\\\xi_{u_{n}}(t_{n})=\xi_{n}^{0}\in\mathcal{X},~t\geq t_{n},
\end{cases}
\end{equation}
such that $\xi_{u_{n}}\rightharpoonup\xi_{u}$ in $\mathcal{C}([-T,\infty);\mathcal{X}_{w})$ for any $T>0$.
\end{theorem}

\begin{proof}
Let $\xi_{u}\in\bar{\mathfrak{E}}((-\infty,\infty))$, and denote $\xi_{u_{n}}=\xi_{u}|_{[t_{n},\infty)}\in\bar{\mathfrak{E}}([t_{n},\infty))$, where $t_{n}\rightarrow-\infty$ as $n\rightarrow\infty$. Clearly, $\xi_{u_{n}}\rightharpoonup\xi_{u}$ in $\mathcal{C}([-T,\infty);\mathcal{X}_{w})$, $\forall T>0$. Since $\xi_{u_{n}}\in\bar{\mathfrak{E}}([t_{n},\infty))$, there exists a sequence $\{\xi_{u_{n}}^{(k)}\}_{k=1}^{\infty}\in\mathfrak{E}([t_{n},\infty))$ such that $\xi_{u_{n}}^{(k)}\rightharpoonup\xi_{u_{n}}$ in $\mathcal{C}([t_{n},\infty);\mathcal{X}_{w})$ as $k\rightarrow\infty$. Applying a standard diagonalization argument, we obtain that there exist a sequence $\xi_{u_{n}}^{(n)}\in\mathfrak{E}([t_{n},\infty))$ (still denoted by $\xi_{u_{n}}$) such that $\xi_{u_{n}}\rightharpoonup\xi_{u}$ in $\mathcal{C}([-T,\infty);\mathcal{X}_{w})$ for any $T>0$. By the definition of $\mathfrak{E}$, we know that $\xi_{u_{n}}$ is the S--S solution of Eq. \eqref{1.1}.

Conversely, let $\xi_{u_{n}}\in\mathfrak{E}([t_{n},\infty))$ and $\xi_{u_{n}}\rightharpoonup\xi_{u}$ in $\mathcal{C}([-T,\infty);\mathcal{X}_{w})$, $\forall T>0$. Consequently, $\{\xi_{u_{n}}|_{[-T,\infty)}:\xi_{u_{n}}\in\mathfrak{E}([t_{n},\infty)\}\subset\mathfrak{E}([-T,\infty))$ converges to $\xi_{u}|_{[-T,\infty)}\in\mathcal{C}([-T,\infty);\mathcal{X}_{w})$. Thus $\xi_{u}\in\bar{\mathfrak{E}}([-T,\infty))$ for any $T>0$. By definition, this implies $\xi_{u}\in\bar{\mathfrak{E}}((-\infty,\infty))$.
\end{proof}

Note that each S--S solution $\xi_{u_{n}}$ can be obtained as a limit of Galerkin approximations (see \cite{ksz,sav,zhou} for more details).
Consequently, the following results can be established, which can be proved using a standard diagonalization argument similar to that in Theorem \ref{ZFT4}.
\begin{corollary}\label{ZFC1}
Assume that the hypotheses of Theorem \ref{ZFT4} are satisfied. For any $\xi_{u}\in\bar{\mathfrak{E}}((-\infty,\infty))$, there exists a sequence $\xi_{u_{k}}^{(k)}$ such that $\xi_{u_{k}}^{(k)}\rightharpoonup\xi_{u}$ in $\mathcal{C}([-T,\infty);\mathcal{X}_{w})$ for any $T>0$. The sequence $u_{k}^{(k)}=\sum_{l=1}^{k}d_{l}^{k}(t)e_{l}$ satisfies the following equation:
\begin{equation}\label{Z3.10}
\begin{cases}
\partial_{t}^{2}u_{k}^{(k)}-\Delta u_{k}^{(k)}+\mathcal{J}\left(\|\partial_{t}u_{k}^{(k)}\|^{2}\right)\partial_{t}u_{k}^{(k)}
+\mathrm{P}_{k}g\left(u_{k}^{(k)}\right)=
\mathrm{P}_{k}h(x),
\\\xi_{u_{k}}^{(k)}(t_{k})=\xi_{k}^{(0)}\in\mathcal{X},~t\geq t_{k},
\end{cases}
\end{equation}
where $t_{k}\rightarrow-\infty$ as $k\rightarrow\infty$. 
Here, $\{e_{k}\}_{i=1}^{\infty}$ denotes an orthonormal system of eigenvectors of the Laplacian $-\Delta$ with Dirichlet boundary conditions, and $\mathrm{P}_{k}$ is the projector from $L^{2}$ onto $E_{k}:=span\{e_{1},e_{2},\cdots,e_{k}\}$.
\end{corollary}

\begin{proposition}\label{ZFP2}
Assume that Assumption \ref{A1.1} is satisfied. Then, for any $\xi_{u}\in\bar{\mathfrak{E}}((-\infty,\infty))$, we have
\begin{align}\label{Z3.11}
\int_{-\infty}^{\infty}\|\partial_{t}u(r)\|^{2p+2}dr\leq\mathcal{Q}(\|h\|^{2}),\quad
\partial_{t}u\in\mathcal{C}_{b}(\mathbb{R},\mathcal{H}^{-\varsigma})\text{ and } \lim\limits_{t\rightarrow\pm\infty}\|\partial_{t}u(t)\|_{\mathcal{H}^{-\varsigma}}=0
\end{align}
for any $0<\varsigma\leq1$, where $\mathcal{Q}(\cdot)$ is a monotone increasing function.
\end{proposition}

\begin{proof}
We may assume without loss of generality that $\mathcal{J}(\cdot)$ satisfies the condition given in \eqref{1.3}. Let $\xi_{u}\in\bar{\mathfrak{E}}((-\infty,\infty))$. By applying Theorem \ref{ZFT4}, we obtain the existence of a sequence of times $t_{n}\rightarrow-\infty$ as $n\rightarrow\infty$ and a sequence of S--S solutions $\xi_{u_{n}}(t)$ of Eq. \eqref{Z3.9} such that $\xi_{u_{n}}\rightharpoonup\xi_{u}$ in $\mathcal{C}([-T,\infty);\mathcal{X}_{w})$ for any $T>0$. By taking the multiplier $\partial_{t}u_{n}+\varepsilon u_{n}$ in \eqref{Z3.9}, we find after some computations that
\begin{align}\label{Z3.12}
\|\xi_{u_{n}}(t)\|_{\mathscr{E}}^{2}+\|u_{n}(t)\|_{L^{q+1}}^{q+1}
\leq e^{-\frac{\varepsilon}{2}(t-s)}\mathcal{Q}(\|\xi_{u_{n}}(s)\|_{\mathscr{E}})+\mathcal{Q}(\|h\|),\quad \forall t\geq s,
\end{align}
where the monotone function $\mathcal{Q}(\cdot)$ and the positive constant $\varepsilon$ are independent of $t$, $s$ and $\xi_{u_{n}}$. Multiplying \eqref{Z3.9} by $\partial_{t}u_{n}$ and integrating, we derive
\begin{align}\label{Z3.13}
2\int_{t}^{\infty}\|\partial_{t}u_{n}(r)\|^{2p+2}dr
\leq\|\xi_{u_{n}}(t)\|_{\mathscr{E}}^{2}+2\langle G(u_{n}(t)),1\rangle+2\langle h,u_{n}(t)\rangle.
\end{align}
Substituting \eqref{Z3.12} into \eqref{Z3.13}, and letting $t=\frac{t_{n}}{2}$ and $s=t_{n}$, we obtain
\begin{align}\label{Z3.14}
2\int_{\frac{t_{n}}{2}}^{\infty}\|\partial_{t}u_{n}(r)\|^{2p+2}dr
\leq e^{\frac{\varepsilon}{2}t_{n}}\mathcal{Q}(\|\xi_{u_{n}}(t_{n})\|_{\mathscr{E}})+\mathcal{Q}(\|h\|).
\end{align}
Recalling that $\xi_{u_{n}}(t_{n})=\xi_{n}^{0}\in\mathcal{X}$, and taking the limit as $n\rightarrow \infty$, we find
\begin{align}\label{Z3.15}
\int_{-\infty}^{\infty}\|\partial_{t}u(r)\|^{2p+2}dr\leq\mathcal{Q}(\|h\|).
\end{align}

To establish convergence for the sequence in \eqref{Z3.11}, consider the sequence
\begin{align*}
\{\xi_{u_{n}}(\cdot)\}_{n=1}^{\infty}
=\{\xi_{u}(\cdot+t_{n})\}_{n=1}^{\infty}\subset
\bar{\mathfrak{E}}((-\infty,\infty)),
\end{align*}
where $t_{n}$ is a sequence tending to $-\infty/\infty$. By Theorem \ref{ZFT2}, $\bar{\mathfrak{E}}((-\infty,\infty))$ is bounded in $\mathcal{C}_{b}(\mathbb{R};\mathscr{E})\cap\mathcal{C}_{b}^{1}(\mathbb{R};\mathscr{E}^{-1})$. Hence, there exists a weakly convergent subsequence (still denoted by $\xi_{u_{n}}$) such that
\begin{align*}
\xi_{u_{n}}\rightharpoonup \xi_{\bar{u}}\quad\text{weakly~in~}L^{2}([T,T+1];\mathscr{E})
\end{align*}
for any $T\in\mathbb{R}$. The weak lower semi-continuity of the norm implies that
\begin{align}\label{Z3.16}
\int_{T}^{T+1}\|\partial_{t}\bar{u}(t)\|^{2p+2}dt
\leq
\liminf\limits_{n\rightarrow\infty}\int_{T}^{T+1}\|\partial_{t}u(t+t_{n})\|^{2p+2}dt.
\end{align}
On the other hand, from the dissipation integral in \eqref{Z3.15}, we have
\begin{align}\label{Z3.17}
\int_{T}^{T+1}\|\partial_{t}u(t+t_{n})\|^{2p+2}dt
=\int_{t_{n}+T}^{t_{n}+T+1}\|\partial_{t}u(t)\|^{2p+2}dt
\rightarrow0,\quad\text{as}~t_{n}\rightarrow-\infty.
\end{align}
Combining \eqref{Z3.16} and \eqref{Z3.17}, we obtain
\begin{align*}
\int_{T}^{T+1}\|\partial_{t}\bar{u}(t)\|^{2p+2}dt=0.
\end{align*}
Thus, $\partial_{t}\bar{u}\equiv0$ on arbitrary $[T,T+1]$. Applying the compact embedding
\begin{align}\label{Z3.18}
\mathcal{C}_{b}(\mathbb{R};\mathscr{E})\cap\mathcal{C}_{b}^{1}(\mathbb{R};\mathscr{E}^{-1})
\Subset\mathcal{C}_{loc}(\mathbb{R};\mathscr{E}^{-\varsigma})\qquad\text{for~any~}0<\varsigma\leq1,
\end{align}
we conclude that
\begin{align}\label{Z3.19}
\partial_{t}u_{n}=\partial_{t}u(t+t_{n})\rightarrow0~\text{in~}\mathcal{C}([T,T+1];\mathcal{H}^{-\varsigma}))~\text{for~any~}T\in\mathbb{R}.
\end{align}
Utilizing \eqref{Z3.15} and \eqref{Z3.19}, we finally conclude
$\lim\limits_{t\rightarrow-\infty}\|\partial_{t}u(t)\|_{\mathcal{H}^{-\varsigma}}=0$.
\end{proof}

\section{Strong attractors}
For the reader's convenience, we briefly review the definition of a (strong) global attractor, see \cite{temam,chueshov1} for more details.
\begin{definition}\label{DS4.1}
Let $S(t)$ be a semigroup acting on a Banach space $\mathcal{Y}$. A set $\mathscr{A}_{s} \subset\mathcal{Y}$ is a (strong) global attractor of $S(t)$ if (a) $\mathscr{A}_{s}$ is compact in $\mathcal{Y}$; (b) $\mathscr{A}_{s}$ is strictly invariant: $S(t) \mathscr{A}_{s}=\mathscr{A}_{s}$; (c) It is an attracting set for the semigroup $S(t)$, i.e., for any bounded set $B \subset\mathcal{Y}$,
$dist_{\mathcal{Y}}(S(t)B,\mathscr{A}_{s}):=\sup\limits_{x\in B}\inf\limits_{y\in\mathscr{A}_{s}}\| S(t)x-y\|_{\mathcal{Y}}\rightarrow0$, as $t\rightarrow\infty$.
\end{definition}
According to the abstract attractor existence theorem in \cite{chueshov1}, the existence of the global attractor can be guaranteed provided the semigroup $S(t)$ is continuous, dissipative and asymptotically compact in $\mathcal{Y}$. While the continuity and dissipativity of $S(t)$ have been established in Section 3, it is necessary to verify its asymptotic compactness.
\subsection{Backward asymptotic regularity}
\begin{theorem}\label{ZFT5}
In addition to the Assumption \ref{A1.1}, suppose that  $J_{0}:=\mathcal{J}(0)>0$. Then, for every complete trajectory $\xi_{u}\in\bar{\mathfrak{E}}((-\infty,\infty))$, there exists a time $T=T(u)$ such that $\xi_{u}\in\mathcal{C}_{b}((-\infty,T];\mathscr{E}^{1})$ and $\|\xi_{u}\|_{\mathcal{C}_{b}((-\infty,T];\mathscr{E}^{1})}\leq\mathcal{Q}(\|h\|^{2})$.
\end{theorem}
\begin{proof}
We will structure the proof into several steps.

\textbf{Step 1}. Rewrite Eq. \eqref{1.1} as follows:
\begin{align*}
\partial_{t}^{2}u-\Delta u+\mathcal{J}(\|\partial_{t}u\|^{2})\partial_{t}u+\ell(-\Delta)^{-1}u+g(u)
=\hat{h}(t):=\ell(-\Delta)^{-1}u+h(x).
\end{align*}
From the definition of $\hat{h}$ and by applying Theorem \ref{ZFT2}, we obtain
\begin{align}\label{ZF4.1}
\|\hat{h}(T)\|^{2}+
\int_{T}^{T+1}\|\partial_{t}\hat{h}(t)\|_{\mathcal{H}^{2}}^{2}dt
\leq \ell^{2}\mathcal{Q}(\|h\|),~\forall T\in\mathbb{R},
\end{align}
where $\mathcal{Q}(\cdot)$ is a monotone function independent of $\ell$ and $T$. Utilizing Proposition \ref{ZFP2}, we infer that
\begin{align}\label{Z4.1}
\partial_{t}\hat{h}\in\mathcal{C}_{b}(\mathbb{R};\mathcal{H}^{2-\varsigma}),\quad
\lim\limits_{t\rightarrow-\infty}\|\partial_{t}\hat{h}(t)\|_{\mathcal{H}^{2-\varsigma}}=0, \quad \forall\varsigma\in(0,1].
\end{align}

\textbf{Step 2}. Applying \cite[Lemma 2.2]{zelik1}, we know that for sufficiently large $\ell=4\kappa_{1}^{2}$, the parabolic equation
\begin{align}\label{Z4.2}
\partial_{t}z-\Delta z+\ell(-\Delta)^{-1}z+g(z)=\hat{h}(t), \quad t\in\mathbb{R}
\end{align}
possesses a unique solution $z(t)$ in the class $C_{b}(\mathbb{R};\mathcal{H}^{2})$ with the following estimates:
\begin{align}\label{Z4.3}
\|z(T)\|_{\mathcal{H}^{2}}\leq \ell^{2}\mathcal{Q}(\|h\|),~
\partial_{t}z\in C_{b}(\mathbb{R};\mathcal{H}^{2}),~
\partial_{t}^{2}z\in L^{2}([T,T+1];\mathcal{H}^{1}),~\forall T\in\mathbb{R},
\end{align}
and the following convergence
\begin{align}\label{Z4.4}
\lim\limits_{T\rightarrow-\infty}\left(\|\partial_{t}z(T)\|_{\mathcal{H}^{2}}
+\|\partial_{t}^{2}z\|_{L^{2}([T,T+1];\mathcal{H}^{1})}\right)=0.
\end{align}

\textbf{Step 3.} Claim \#1: There exists a time $T=T(u,\ell)$ such that the problem
\begin{align}\label{Z4.5}
\partial_{t}^{2}v-\Delta v+\mathcal{J}(\|\partial_{t}v\|^{2})\partial_{t}v+g(v)+\ell(-\Delta)^{-1}v=\hat{h}(t),\quad t\leq T
\end{align}
has a unique regular backward solution $\xi_{v}(t)\in\mathscr{E}^{1}$, which satisfies the following estimate:
\begin{align}\label{Z4.6}
\|\partial_{t}v(t)\|_{\mathcal{H}^{2}}+\|v(t)\|_{\mathcal{H}^{2}}
\leq\mathcal{Q}_{\ell}(\|h\|),\quad t\leq T,
\end{align}
for some monotone function $\mathcal{Q}_{\ell}(\cdot)$ depending on $\ell$. Furthermore, we also have
\begin{align}\label{Z4.7}
\lim_{t\rightarrow-\infty}\|\partial_{t}v(t)\|_{L^{\infty}}=0.
\end{align}

\emph{Proof of claim \#1}: We divide the proof into two essential steps.

\textbf{Step 3(i).} Assume $v=z+W$, then $W$ satisfies
\begin{align}\label{Z4.8}
\nonumber\partial_{t}^{2}W&-\Delta W+\mathcal{J}(\|\partial_{t}(z+W)\|^{2})\partial_{t}(z+W)-
\mathcal{J}(\|\partial_{t}z\|^{2})\partial_{t}z
\\&+g(z+W)-g(z)+\ell(-\Delta)^{-1}W=H_{z}(t):=-\partial_{t}^{2}z+\left(1-\mathcal{J}(\|\partial_{t}z\|^{2})\right)\partial_{t}z.
\end{align}
We can apply the implicit function theorem to solve Eq. \eqref{Z4.8} in the space
\begin{align}\label{Z4.9}
\mathds{W}_{T}:=C_{b}((-\infty,T];\mathscr{E}^{1}),
\end{align}
where $T$ is sufficiently small.

Firstly, recalling \textbf{Step 2}, we have
\begin{align}\label{Z4.10}
\lim_{T\rightarrow-\infty}
\|H_{z}\|_{L^{2}([T,T+1];\mathcal{H}^{1})}=0.
\end{align}
Now, we intend to verify that the variation equation at $W=0$
\begin{align}\label{Z4.11}
\nonumber&\partial_{t}^{2}V-\Delta V+g'(z)V+\ell(-\Delta)^{-1}V
\\&\qquad+\mathcal{J}(\|\partial_{t}z\|^{2})\partial_{t}V
+2\mathcal{J}'(\|\partial_{t}z\|^{2})\langle \partial_{t}z,\partial_{t}V\rangle\partial_{t}z=H(t)
\end{align}
is uniquely solvable for every $H\in L^{2}_{loc}((-\infty,T];\mathcal{H}^{1})$ such that
\begin{align*}
\|H\|_{L_{b}^{2}((-\infty,T];\mathcal{H}^{1})}:=\sup_{t\in(-\infty,T-1)}\|H\|_{L^{2}((t,t+1];\mathcal{H}^{1})}<\infty
\end{align*}
provided that $T$ is small enough. Taking the multiplier $\partial_{t}V+2\varepsilon V$ in \eqref{Z4.11} yields
\begin{align}\label{ZF4.13}
\frac{d}{dt}\mathcal{E}_{V}(t)+\mathcal{Q}_{V}(t)+\mathcal{J}_{V}(t)+\mathcal{G}_{V}(t)=\mathcal{I}_{V}(t),
\end{align}
where
\begin{align*}
\mathcal{E}_{V}(t)&=\|\partial_{t}V\|^{2}
+\|V\|_{\mathcal{H}^{1}}^{2}
+\ell\|V\|_{\mathcal{H}^{-1}}^{2}
+2\varepsilon\langle\partial_{t}V,V\rangle
+\langle g'(z)V,V\rangle,
\\\mathcal{Q}_{V}(t)&=\left(2\mathcal{J}(\|\partial_{t}z\|^{2})-2\varepsilon\right)\|\partial_{t}V\|^{2}
+2\varepsilon\|V\|_{\mathcal{H}^{1}}^{2}
\\&\qquad+2\varepsilon\ell\|V\|_{\mathcal{H}^{-1}}^{2}+2\varepsilon\langle g'(z)V,V\rangle,
\\\mathcal{J}_{V}(t)&=4\mathcal{J}'(\|\partial_{t}z\|^{2})\langle \partial_{t}z,\partial_{t}V\rangle^{2}
+2\varepsilon\mathcal{J}(\|\partial_{t}z\|^{2})\langle\partial_{t}V,V\rangle
\\&\qquad+4\varepsilon\mathcal{J}'(\|\partial_{t}z\|^{2})\langle \partial_{t}z,\partial_{t}V\rangle\langle\partial_{t}z,V\rangle,
\\\mathcal{G}_{V}(t)&=-\langle g''(z)\partial_{t}z,V^{2}\rangle,
\qquad\mathcal{I}_{V}(t)=2\langle H(t),\partial_{t}V+\varepsilon V\rangle.
\end{align*}
By choosing
\begin{align}\label{ZF4.14}
\varepsilon=\min\{\frac{1}{4},\frac{J_{0}}{6},\frac{\lambda_{1}}{8},\frac{\lambda_{1}J_{0}}{18},
\frac{\lambda_{1}J_{0}}{(2J_{0}+4\mathcal{J}'(0)+1)^{2}}\}
\end{align}
sufficiently small, we ensure that
\begin{align}\label{ZF4.15}
\frac{1}{2}\|\xi_{V}(t)\|_{\mathscr{E}}^{2}
\leq\mathcal{E}_{V}(t)
\leq C_{0}\|\xi_{V}(t)\|_{\mathscr{E}}^{2}
\end{align}
and
\begin{align}\label{ZF4.16}
\frac{d}{dt}\mathcal{E}_{V}(t)
+\varepsilon\mathcal{E}_{V}(t)
\leq C_{1}\|H(t)\|^{2}+\mathcal{Q}_{V}^{(0)}(t),
\end{align}
where $C_{0}$ depend on $\|h\|$ and $\kappa_{1}$, $C_{1}=\frac{2}{J_{0}}+\frac{8}{\lambda_{1}}$ and
\begin{align}\label{ZF4.17}
\mathcal{Q}_{V}^{(0)}(t)=-\frac{J_{0}}{2}\|\partial_{t}V\|^{2}
-\frac{\varepsilon}{4}\|V\|_{\mathcal{H}^{1}}^{2}
+\frac{\|g''(z)\|_{L^{\infty}}}{\lambda_{1}}\|\partial_{t}z\|_{L^{\infty}}\|V\|_{\mathcal{H}^{1}}^{2}
+\sum_{i=1}^{2}\mathcal{J}_{V}^{(i)}(t),
\end{align}
and
\begin{align}\label{ZF4.18}
\mathcal{J}_{V}^{(1)}=\frac{2\varepsilon}{\sqrt{\lambda_{1}}}
\mathcal{J}(\|\partial_{t}z\|^{2})\|\partial_{t}V\|\|V\|_{\mathcal{H}^{1}},\quad
\mathcal{J}_{V}^{(2)}=\frac{4\varepsilon}{\sqrt{\lambda_{1}}}
\mathcal{J}'(\|\partial_{t}z\|^{2})\|\partial_{t}z\|^{2}\|\partial_{t}V\|\|V\|_{\mathcal{H}^{1}}.
\end{align}
Applying estimate \eqref{Z4.3}, convergence \eqref{Z4.4}, and the embedding $\mathcal{H}^{2}\subset L^{\infty}(\bar{\Omega})$, we obtain
\begin{align}\label{ZF4.19}
\mathcal{Q}_{V}^{(0)}(t)\leq
-2J_{0}\|\partial_{t}V\|^{2}-\frac{\varepsilon}{8}\|V\|_{\mathcal{H}^{1}}^{2}+C_{2}\|\partial_{t}V\|\|V\|_{\mathcal{H}^{1}}
\end{align}
with $C_{2}=\frac{4\varepsilon}{\sqrt{\lambda_{1}}}(\mathcal{J}'(0)+\frac{1}{8})+\frac{2\varepsilon}{\sqrt{\lambda_{1}}}(J_{0}+\frac{1}{4})$. Applying \eqref{ZF4.14} to \eqref{ZF4.19}, applying Gronwall's inequality to \eqref{ZF4.16}, we deduce that
\begin{align}\label{ZF4.20}
\|\xi_{V}(t)\|_{\mathscr{E}}^{2}
\leq C_{1}\int_{-\infty}^{t}e^{-\varepsilon(t-r)}\|H(r)\|^{2}dr
\leq C_{1}\varepsilon^{-1}\|H\|_{L_{b}^{2}((-\infty,T];L^{2})},\quad t\leq T.
\end{align}
It follows that the solution to \eqref{Z4.11} is unique.

Taking the multiplier $-\Delta(\partial_{t}V+\varepsilon V)$ in \eqref{Z4.11}, and choosing
\begin{align}\label{ZF4.21}
\varepsilon=\min\{\frac{1}{4},\frac{\lambda_{1}}{8},\frac{J_{0}}{6},\frac{\lambda_{1}J_{0}}{16},
\frac{\lambda_{1}J_{0}}{16(\sqrt{\lambda_{1}}J_{0}+4\mathcal{J}'(0)+1)^{2}}\}
\end{align}
small enough, to discover after some computations that
\begin{align}\label{ZF4.22}
\frac{1}{4}\|\xi_{V}(t)\|_{\mathscr{E}^{1}}^{2}
\leq
\tilde{\mathcal{E}}_{V}(t)
\leq C_{2}\|\xi_{V}(t)\|_{\mathscr{E}^{1}}^{2}
\end{align}
and
\begin{align}\label{ZF4.23}
\frac{d}{dt}\tilde{\mathcal{E}}_{V}(t)+\varepsilon\tilde{\mathcal{E}}_{V}(t)
\leq C_{3}\|H(t)\|_{\mathcal{H}^{1}}^{2}+C_{4}\|V(t)\|_{\mathcal{H}^{1}}^{2}
+\sum_{i=1}^{2}\tilde{\mathcal{Q}}_{V}^{(i)}(t),
\end{align}
where the positive constants $C_{i}$ ($i=2,3,4$) depend on $\lambda_{1}$, $\|h\|$, and the structural parameters in model \eqref{1.1}. The $\tilde{\mathcal{E}}_V(t)$ and $\tilde{\mathcal{Q}}_V^{(i)}(t)$ are defined as follows:
\begin{align*}
\tilde{\mathcal{E}}_{V}(t)&=\|\partial_{t}V\|_{\mathcal{H}^{1}}^{2}+\|V\|_{\mathcal{H}^{2}}^{2}
+\ell\|V\|^{2}
+2\varepsilon\langle\langle\partial_{t}V,V\rangle\rangle
+\langle g'(z)\nabla V,\nabla V\rangle,
\\\tilde{\mathcal{Q}}_{V}^{(1)}(t)&=
-\frac{J_{0}}{2}\|\partial_{t}V\|_{\mathcal{H}^{1}}^{2}
-\frac{\varepsilon}{8}\|V\|_{\mathcal{H}^{2}}^{2}
\\&\qquad+
\left(\frac{2\varepsilon}{\sqrt{\lambda_{1}}}\mathcal{J}(\|\partial_{t}z\|^{2})
+\frac{4\varepsilon\mathcal{J}'(\|\partial_{t}z\|^{2})}{\lambda_{1}}\right)\|\partial_{t}V\|_{\mathcal{H}^{1}}\|V\|_{\mathcal{H}^{2}},
\\
\tilde{\mathcal{Q}}_{V}^{(2)}(t)
&=-\frac{J_{0}}{2}\|\partial_{t}V\|_{\mathcal{H}^{1}}^{2}
-\frac{3\varepsilon}{4}\|V\|_{\mathcal{H}^{2}}^{2}
+\frac{1}{\lambda_{1}}\|g'(z)\|_{L^{\infty}}\|\partial_{t}z\|_{L^{\infty}}\|V\|_{\mathcal{H}^{2}}^{2}
\\&\qquad+\frac{4\mathcal{J}'(\|\partial_{t}z\|^{2})}{\sqrt{\lambda_{1}}}\|\partial_{t}z\|_{L^{\infty}}\|\partial_{t}z\|_{\mathcal{H}^{1}}
\|\partial_{t}V\|_{\mathcal{H}^{1}}^{2}.
\end{align*}
Convergence \eqref{Z4.4}, along with estimates \eqref{Z4.3} and \eqref{ZF4.20}, and the condition in \eqref{ZF4.21}, imply that
\begin{align}\label{ZF4.24}
\frac{d}{dt}\tilde{\mathcal{E}}_{V}(t)+\varepsilon\tilde{\mathcal{E}}_{V}(t)
\leq C_{3}\|H(t)\|_{\mathcal{H}^{1}}^{2}+C_{4}\|H\|_{L_{b}^{2}((-\infty,T];L^{2})}^{2},\quad t\leq T
\end{align}
with $T$ is small enough. Recalling \eqref{ZF4.22}, we apply Gronwall's inequality to \eqref{ZF4.24} once more, leading to the result:
\begin{align}\label{ZF4.25}
\|\xi_{V}(t)\|_{\mathscr{E}^{1}}^{2}
\leq C_{5}\|H\|_{L_{b}^{2}((-\infty,T];\mathcal{H}^{1})}^{2},\quad t\leq T.
\end{align}
Applying the implicit function theorem to Eq. \eqref{Z4.8}, we obtain a unique solution $\xi_{W}\in\mathds{W}_{T}$ of Eq. \eqref{Z4.8} with $T$ sufficiently small. Moreover, combining \eqref{ZF4.1}, \eqref{Z4.3}, and \eqref{Z4.10}, we find that
\begin{align}\label{ZF4.26}
\|\xi_{W}(t)\|_{\mathscr{E}^{1}}^{2}
\leq \mathcal{Q}(\|h\|),\quad t\leq T \quad \text{\rm and }
\quad
\lim\limits_{t\rightarrow-\infty}\|\partial_{t}W(t)\|_{\mathcal{H}^{1}}=0,
\end{align}
where the monotone function independent of $t$, $T$ and $W$. By combining the estimates for $z$ in \textbf{Step 2} with \eqref{ZF4.26}, we derive
\begin{align}\label{ZF4.27}
\|\partial_{t}v(t)\|_{\mathcal{H}^{1}}^{2}+\|v(t)\|_{\mathcal{H}^{2}}^{2}
\leq\mathcal{Q}(\|h\|^{2}), ~t\leq T(\ell,u) \text{\quad and \quad}
\lim\limits_{t\rightarrow-\infty}\|\partial_{t}v(t)\|_{\mathcal{H}^{1}}=0.
\end{align}

\textbf{Step 3(ii).} It remains to check the estimates of $\partial_{t}v$ in \eqref{Z4.6} and the convergence in \eqref{Z4.7}. Define $\zeta=\partial_{t}v$ and differentiate \eqref{Z4.5} with respect to $t$, yielding the equation
\begin{align}\label{ZF4.28}
\partial_{t}^{2}\zeta-\Delta\zeta+\mathcal{J}(\|\partial_{t}v\|^{2})\partial_{t}\zeta
+2\mathcal{J}'(\|\partial_{t}v\|^{2})\langle\zeta,\partial_{t}\zeta\rangle\zeta
+\ell(-\Delta)^{-1}\zeta=I_{\zeta}(t)
\end{align}
with $I_{\zeta}(t)=\partial_{t}\hat{h}(t)-g'(v)\zeta$. Taking the multiplier $-\Delta(\partial_{t}\zeta+\varepsilon\zeta)$ in \eqref{ZF4.28}, we obtain
\begin{align}\label{ZF4.29}
\frac{d}{dt}\mathcal{E}_{\zeta}(t)
+\mathcal{Q}_{\zeta}(t)
+\mathcal{J}_{\zeta}(t)=\mathcal{I}_{\zeta}(t),
\end{align}
where
\begin{align*}
\mathcal{E}_{\zeta}(t)&=
\|\partial_{t}\zeta\|_{\mathcal{H}^{1}}^{2}
+\|\zeta\|_{\mathcal{H}^{2}}^{2}+\ell\|\zeta\|^{2}
+2\varepsilon\langle\langle\partial_{t}\zeta,\zeta\rangle\rangle,~
\mathcal{I}_{\zeta}(t)=2\langle\langle I_{\zeta}(t),\partial_{t}\zeta+\varepsilon\zeta\rangle\rangle,
\\\mathcal{Q}_{\zeta}(t)&=(2\mathcal{J}(\|\partial_{t}v\|^{2})-2\varepsilon)
\|\partial_{t}\zeta\|_{\mathcal{H}^{1}}^{2}
+2\varepsilon\|\zeta\|_{\mathcal{H}^{2}}^{2}
+2\varepsilon\ell\|\zeta\|^{2},
\\\mathcal{J}_{\zeta}(t)&=
4\mathcal{J}'(\|\partial_{t}v\|^{2})\langle\zeta,\partial_{t}\zeta\rangle
\langle\langle\zeta,\partial_{t}\zeta\rangle\rangle
+4\varepsilon\mathcal{J}'(\|\partial_{t}v\|^{2})\langle\zeta,\partial_{t}\zeta\rangle
\|\zeta\|_{\mathcal{H}^{1}}^{2}.
\end{align*}
Using \eqref{ZF4.27} and similar calculations as in the proof of \eqref{ZF4.24}, we can select
\begin{align*}
\varepsilon=\min\{1,\frac{\sqrt{\lambda_{1}}}{2},\frac{J_{0}\lambda_{1}}{4},\frac{J_{0}}{3},\frac{J_{0}\lambda_{1}^{2}}{64(\mathcal{J}'(0)+1)^{2}}\}
\end{align*}
to be sufficiently small such that
\begin{align*}
\frac{1}{2}\|\xi_{\zeta}(t)\|_{\mathscr{E}^{1}}^{2}
\leq \|\mathcal{E}_{\zeta}(t)\|
\leq C\|\xi_{\zeta}(t)\|_{\mathscr{E}^{1}}^{2}
\end{align*}
and
\begin{align}\label{ZF4.30}
\frac{d}{dt}\mathcal{E}_{\zeta}(t)
+\varepsilon\mathcal{E}_{\zeta}(t)
\leq C \|I_{\zeta}(t)\|_{\mathcal{H}^{1}}^{2},\quad t\leq T.
\end{align}
Applying Gronwall's inequality to \eqref{ZF4.30}, we obtain
\begin{align}\label{ZF4.31}
\|\xi_{\zeta}(t)\|_{\mathscr{E}^{1}}^{2}
\leq C\int_{-\infty}^{t}e^{-\varepsilon(t-r)}\|I_{\zeta}(r)\|_{\mathcal{H}^{1}}^{2}dr,\quad t\leq T.
\end{align}
By embedding $\mathcal{H}^2 \subset \mathcal{C}(\bar{\Omega})$, using the convergence in \eqref{Z4.1} and \eqref{ZF4.27}, and the estimates in \eqref{ZF4.31}, we derive the estimates for the $\mathcal{H}^2$--norm of $\partial_t v(t)$ in \eqref{Z4.6} and convergence in \eqref{Z4.7}.

\textbf{Step 4.} To establish $u\equiv v$ for $t\leq T$, consider that the complete trajectory $\xi_{u}\in\bar{\mathfrak{E}}((-\infty,\infty))$. By applying Corollary \ref{ZFC1}, there exists a sequence $\xi_{u_{k}}^{(k)}$ such that $\xi_{u_{k}}^{(k)}\rightharpoonup\xi_{u}$ in $\mathcal{C}([-T,\infty);\mathcal{X}_{w})$ for any $T>0$. Furthermore, $u_{k}^{(k)}=\sum_{l=1}^{k}d_{l}^{k}(t)e_{l}$ satisfies the equation
\begin{align}\label{ZF4.32}
\partial_{t}^{2}u_{k}^{(k)}-\Delta u_{k}^{(k)}+\mathcal{J}\left(\|\partial_{t}u_{k}^{(k)}\|^{2}\right)\partial_{t}u_{k}^{(k)}
+\mathrm{P}_{k}g\left(u_{k}^{(k)}\right)=
\mathrm{P}_{k}h(x)
\end{align}
with the initial condition $\xi_{u_{k}}^{(k)}(t_{k})=\mathrm{P}_{k}\xi_{u_{k}}(t_{k})$, where $t\geq t_{k}$ and $t_{k}\rightarrow-\infty$ as $k\rightarrow\infty$, and $\|\xi_{u_{k}}^{(k)}(t_{k})\|_{\mathscr{E}}\leq C$. Define $v_{k}(t)=\mathrm{P}_{k}v(t)$, $t\leq T$. By \textbf{Step 3}, the solution $\xi_{v}(t)$ is bounded in $\mathscr{E}^{1}$ for $t\leq T$, and consequently
\begin{align}\label{ZF4.33}
\lim\limits_{k\rightarrow\infty}\|\xi_{v_{k}}-\xi_{v}\|_{\mathcal{C}_{b}((-\infty,t],\mathscr{E})}=0,~t\leq T,\quad
\lim\limits_{k\rightarrow\infty}\|\xi_{v_{k}}-\xi_{v}\|_{\mathcal{C}_{b}((-\infty,T]\times\Omega)}=0.
\end{align}
Here, we utilized the fact that $\mathcal{H}^{2}\Subset\mathcal{C}(\overline{\Omega})$ and that Fourier series converge uniformly on compact sets. Define $Z(t):=u(t)-v(t)$ and $Z_{k}(t):=u_{k}^{(k)}(t)-v_{k}(t)$. From equation \eqref{ZF4.32}, we obtain the following equation
\begin{align}\label{ZF4.34}
\nonumber\partial_{t}^{2}Z_{k}&-\Delta Z_{k}+\ell(-\Delta)^{-1}Z_{k}
+\Gamma_{2}(t)\partial_{t}Z_{k}
\\&+\Gamma_{1}(t)(\|\partial_{t}u_{k}^{(k)}\|^{2}-\|\partial_{t}v_{k}\|^{2})
(\partial_{t}u_{k}^{(k)}+\partial_{t}v_{k})
+\mathrm{P}_{k}(g(u_{k}^{(k)})-g(v_{k}))
=G_{k}(t),
\end{align}
where
\begin{align}\label{ZHOU4.35}
\begin{split}
\Gamma_{2}(t)&=\frac{\mathcal{J}\left(\|\partial_{t}u_{k}^{(k)}\|^{2}\right)
+\mathcal{J}\left(\|\partial_{t}v_{k}\|^{2}\right)}{2},
\\
\Gamma_{1}(t)&=\frac{1}{2}\int_{0}^{1}
\mathcal{J}'\left(s\|\partial_{t}u_{k}^{(k)}(t)\|^{2}+(1-s)\|\partial_{t}v_{k}(t)\|^{2}\right)ds,
\\
G_{k}(t)&=
\mathrm{P}_{k}g(v)-\mathrm{P}_{k}g(v_{k})
+\mathcal{J}(\|\partial_{t}v\|^{2})\partial_{t}v_{k}
-\mathcal{J}(\|\partial_{t}v_{k}\|^{2})\partial_{t}v_{k}.
\end{split}
\end{align}
In addition, by convergence as described in equation \eqref{ZF4.33}, we have
\begin{align}\label{ZHF4.34}
\lim\limits_{k\rightarrow\infty}\|G_{k}\|_{\mathcal{C}_{b}((-\infty,T]\times\Omega)}=0,\quad\text{\rm and}\quad
\|\xi_{Z_{k}}(t_{k})\|_{\mathscr{E}}\leq C
\end{align}
with $C$ independent of $k$. Multiplying Eq. \eqref{ZF4.34} by $\partial_{t}Z_{k}+\varepsilon Z_{k}$ and setting
\begin{align*}
\mathcal{E}_{Z_{k},v_{k}}(t)=\|\xi_{Z_{k}}\|_{\mathscr{E}}^{2}
+2\varepsilon\langle\partial_{t}Z_{k},Z_{k}\rangle
+\ell\|Z_{k}\|_{\mathcal{H}^{-1}}^{2}
+2\langle G(v_{k}+Z_{k})-G(v_{k})-g(v_{k})Z_{k},1\rangle,
\end{align*}
to derive the identity
\begin{align}\label{ZHOU4.37}
\frac{d}{dt}\mathcal{E}_{Z_{k},v_{k}}(t)
+\varepsilon\mathcal{E}_{Z_{k},v_{k}}(t)+\sum_{i=1}^{2}\mathcal{Q}^{(i)}_{Z_{k},v_{k}}(t)
=\sum_{i=1}^{3}\mathcal{G}^{(i)}_{Z_{k},v_{k}}(t),
\end{align}
where
\begin{align}\label{ZHOU4.38}
\begin{split}
\mathcal{Q}^{(1)}_{Z_{k},v_{k}}(t)&=
(2\Gamma_{2}(t)-3\varepsilon)\|\partial_{t}Z_{k}\|^{2}
+2\Gamma_{1}(t)(\|\partial_{t}u_{k}^{(k)}\|^{2}-\|\partial_{t}v_{k}\|^{2})^{2}
\\&\qquad+\varepsilon\|Z_{k}\|_{\mathcal{H}^{1}}^{2}
+\varepsilon\ell\|Z_{k}\|_{\mathcal{H}^{-1}}^{2}
-2\varepsilon^{2}\langle\partial_{t}Z_{k},Z_{k}\rangle,
\\
\mathcal{Q}^{(2)}_{Z_{k},v_{k}}(t)&=
2\varepsilon\Gamma_{1}(t)(\|\partial_{t}u_{k}^{(k)}\|^{2}-\|\partial_{t}v_{k}\|^{2})
\langle\partial_{t}u_{k}^{(k)}+\partial_{t}v_{k},Z_{k}\rangle
\\&\qquad+2\varepsilon\Gamma_{2}(t)\langle\partial_{t}Z_{k},Z_{k}\rangle,
\\
\mathcal{G}^{(1)}_{Z_{k},v_{k}}(t)&=
2\varepsilon\langle G(v_{k}+Z_{k})-G(v_{k})-g(v_{k})Z_{k}
\\&\qquad-[g(v_{k}+Z_{k})-g(v_{k})]Z_{k},1\rangle,
\\\mathcal{G}^{(2)}_{Z_{k},v_{k}}(t)&=
2\langle g(v_{k}+Z_{k})
-g(v_{k})-g'(v_{k})Z_{k},\partial_{t}v_{k}\rangle,
\\\mathcal{G}^{(3)}_{Z_{k},v_{k}}(t)&=2\langle G_{k},\partial_{t}Z_{k}+\varepsilon Z_{k}\rangle.
\end{split}
\end{align}
Using assumptions \eqref{1.4} and \eqref{1.5}, we can derive the following inequalities:
\begin{align}\label{Zfeng4.39}
G(v+z)-G(v)-[g(v)z+(g(v+z)-g(v))z]
\leq\frac{\kappa_{1}}{2}|z|^2-\delta_q'|z|^2\left(|v|^{q-1}+|z|^{q-1}\right)
\end{align}
and
\begin{align}\label{Zfeng4.40}
|g(v+z)-g(v)-g'(v)z|\leq C|z|^{2}\left(1+|v|^{q-2}+|z|^{q-2}\right),
\end{align}
where $\delta_q^{\prime}$ is positive and depends only on $q$, and the constant $C$ is independent of $v$ and $z$ (see \cite[Proposition 2.1]{zelik1} for more details). Applying \eqref{Zfeng4.39} and \eqref{Zfeng4.40} to \eqref{ZHOU4.38}, we obtain
\begin{align*}
\mathcal{G}^{(1)}_{Z_{k},v_{k}}(t)
\leq\varepsilon\kappa_{1}\|Z_{k}\|^{2}-\varepsilon C(\kappa_{2},q)\|Z_{k}\|_{L^{q+1}}^{q+1}
\end{align*}
and
\begin{align*}
\mathcal{G}^{(2)}_{Z_{k},v_{k}}(t)
\leq C\|\partial_{t}v_{k}\|_{L^{\infty}}\langle |Z_{k}|^{2}(1+|v_{k}|^{q-2}+|Z_{k}|^{q-2}),1\rangle,
\end{align*}
where the positive constant $C_{6}=2^{p-2}C_{g}$. Here, the Taylor--MacLaurin formula and assumption \eqref{1.4} have been implicitly used. Choosing
\begin{align}\label{ZF4.35}
\varepsilon=\min\{1,\frac{J_{0}}{6},\frac{\lambda_{1}}{4},\frac{3J_{0}\lambda_{1}}{16\mathcal{J}^{2}(R_{0}^{2})},
\frac{\lambda_{1}J_{0}}{16^{2}(J_{0}+M_{0}R_{0}^{2}+1)^{2}}\},
\end{align}
where $M_{0}:=\sup\limits_{0\leq r\leq 2R_{0}^{2}}\mathcal{J}'(r)$, we have
\begin{align}\label{ZF4.36}
\frac{d}{dt}\mathcal{E}_{Z_{k},v_{k}}(t)
+\varepsilon\mathcal{E}_{Z_{k},v_{k}}(t)
\leq \mathcal{G}_{Z_{k},v_{k}}(t),
\end{align}
where
\begin{align*}
\mathcal{G}_{Z_{k},v_{k}}(t)&=
C_{7}\|G_{k}(t)\|^{2}
-\varepsilon C(\kappa_{2},q)\|Z_{k}\|_{L^{q+1}}^{q+1}-\mathcal{Q}^{(3)}_{Z_{k},v_{k}}(t)+
\\&\qquad+C_{6}\|\partial_{t}v_{k}\|_{L^{\infty}}\langle |Z_{k}|^{2}(1+|v_{k}|^{q-2}+|Z_{k}|^{q-2}),1\rangle,
\\
\mathcal{Q}^{(3)}_{Z_{k},v_{k}}(t)&=\frac{9J_{0}}{16}\|\partial_{t}Z_{k}\|^{2}
+\frac{3\varepsilon}{16}\|Z_{k}\|_{\mathcal{H}^{1}}^{2}
-\varepsilon[\mathcal{J}(\|\partial_{t}v_{k}\|^{2})+
\\&\qquad+\frac{4\Gamma_{1}(t)}{\sqrt{\lambda_{1}}}(R_{0}^{2}+\|\partial_{t}v_{k}\|^{2})]
\|\partial_{t}Z_{k}\|\|Z_{k}\|_{\mathcal{H}^{1}}
\end{align*}
and  $C_{7}=\frac{8}{3J_{0}}+\frac{8}{\lambda_{1}}$. According to the convergence results given in \eqref{Z4.7} and \eqref{ZF4.33}, and considering our choice of $\varepsilon$ as specified in \eqref{ZF4.35}, there exists a time $T^{\prime} \leq T$ such that, for sufficiently large $k$, we obtain
\begin{align*}
\mathcal{G}_{Z_{k},v_{k}}(t)
\leq C_{7}\|G_{k}(t)\|^{2},\qquad t\leq T'.
\end{align*}
Applying Gronwall's inequality to \eqref{ZF4.36}, we obtain
\begin{align}
\mathcal{E}_{Z_{k},v_{k}}(t)
\leq \mathcal{E}_{Z_{k},v_{k}}(t_{k})e^{-\varepsilon(t-t_{k})}
+C_{7}\int_{t_{k}}^{t}e^{-\varepsilon(t-r)}\|G_{k}(r)\|^{2}dr, \quad t\leq T'.
\end{align}
Utilizing the fact that (see \cite[Proposition 2.1]{zelik1} for details)
\begin{align*}
G(v+z)-G(v)-g(v)z\geq-\kappa_{1}|z|^{2}+\delta_{q}|z|^{2}(|v|^{q-1}+|z|^{q-1}),
\end{align*}
and combining \eqref{1.4}, $\ell\geq4\kappa_{1}^{2}$ and \eqref{ZF4.35}, we derive
\begin{align}\label{ZHF4.39}
\nonumber\|\xi_{Z_{k}}(t)\|_{\mathscr{E}}^{2}
\leq &\mathcal{Q}(\|\xi_{u_{k}}^{(k)}(t_{k})\|_{\mathscr{E}}^{2},\|\xi_{v_{k}}(t_{k})\|_{\mathscr{E}}^{2})
e^{-\varepsilon(t-t_{k})}
\\&+2C_{7}\int_{t_{k}}^{t}e^{-\varepsilon(t-r)}\|G_{k}(r)\|^{2}dr,~t\leq T',
\end{align}
where the monotone function $\mathcal{Q}(\cdot,\cdot)$ is independent of $Z_{k}$, $v_{k}$, $k$, $t$ and $t_{k}$. Noting that $\|\xi_{u_{k}}^{(k)}(t_{k})\|_{\mathscr{E}}$ is uniformly bounded and \eqref{ZHF4.34}, we take the limit as $k\rightarrow \infty$ in \eqref{ZHF4.39}, thereby obtaining the estimate $\|\xi_{Z}(t)\|_{\mathscr{E}}^{2}\leq 0$, $t\leq T'$. Thus, the proof of Theorem \ref{ZFT5} is complete.
\end{proof}

\begin{remark}
Using the fact that $u(t)=v(t)$ for $t\leq T_{u}$ and the estimate given in \eqref{Z4.6}, we obtain
\begin{align}\label{ZHOU4.43}
\|\partial_{t}u(t)\|_{\mathcal{H}^{2}}+\|u(t)\|_{\mathcal{H}^{2}}
\leq\mathcal{Q}(\|h\|),\quad t\leq T_{u},
\end{align}
where the monotone function $\mathcal{Q}$ depends only on the structural parameters specified in Assumption \ref{A1.1}.
\end{remark}

\begin{theorem}\label{ZFT6}
Let $\mathcal{J}(\cdot)$, $g$ and $h$ satisfy Assumption \ref{A1.1}, and let $\mathcal{J}(0)>0$. Then, the weak global attractor $\mathscr{A}_{w}$ for ES $\mathfrak{E}$, as established in Theorem \ref{ZFT3}, is in a more regular space: $\mathscr{A}_{w}\subset\mathscr{E}^{1}$.
\end{theorem}
\begin{proof}
Let $\xi_{u}$ denote the complete trajectory of equation \eqref{1.1}. By applying Theorem \ref{ZFT5}, there exists a time $T_{0}$ such that $\xi_{u}(t)\in\mathscr{E}^{1}$ for all $t\leq T_{0}$. According to Theorem \ref{ZFT1}, there exists an extension $\bar{u}$ for $t\geq T_{0}$ such that $\bar{u}(t)=u(t)$ for $t\leq T_{0}$ and $\bar{u}(t)$ is a S--S solution of Eq. \eqref{1.1} for all $t\in\mathbb{R}$. Consequently, it follows that $\xi_{\bar{u}}(t)\in\mathscr{E}^{1}$ for all $t\in\mathbb{R}$.

We aim to show that $\xi_{\bar{u}}(t)=\xi_{u}(t)$ for all $t\in\mathbb{R}$. Since $\xi_{u}\in\bar{\mathfrak{E}}((-\infty,\infty))$, we apply Corollary \ref{ZFC1} to deduce
\begin{align}\label{ZHF4.40}
\partial_{t}^{2}u_{k}^{(k)}-\Delta u_{k}^{(k)}+\mathcal{J}\left(\|\partial_{t}u_{k}^{(k)}\|^{2}\right)\partial_{t}u_{k}^{(k)}
+\mathrm{P}_{k}g\left(u_{k}^{(k)}\right)=
\mathrm{P}_{k}h,~
\xi_{u_{k}}^{(k)}(t_{k})=\xi_{k}^{(0)}\in\mathcal{X},
\end{align}
where $t\geq t_{k}$ and $\lim\limits_{k\rightarrow\infty}t_{k}=-\infty$. Clearly, $\bar{u}_{k}=\mathrm{P}_{k}\bar{u}$ satisfies
\begin{align}\label{ZHF4.41}
\partial_{t}^{2}\bar{u}_{k}-\Delta \bar{u}_{k}+\mathcal{J}(\|\partial_{t}\bar{u}\|^{2})\partial_{t}\bar{u}_{k}+\mathrm{P}_{k}g(\bar{u})=\mathrm{P}_{k}h,
\quad\xi_{\bar{u}_{k}}(t_{k})=\mathrm{P}_{k}\xi_{\bar{u}}(t_{k}).
\end{align}
and
\begin{align}\label{ZHOU4.46}
\lim\limits_{k\rightarrow\infty}\|\xi_{\bar{u}_{k}}-\xi_{\bar{u}}\|_{\mathcal{C}_{b}((-\infty,t];\mathscr{E})}=0,
~\forall t\leq T_{1},\quad
\lim\limits_{k\rightarrow\infty}\|\xi_{\bar{u}_{k}}-\xi_{\bar{u}}\|_{\mathcal{C}_{b}((-\infty,T_{1}]\times\Omega)}=0
\end{align}
with $T_{0}<T_{1}$. Let $Z(t)=u(t)-\bar{u}(t)$, $Z_{k}(t)=u_{k}^{(k)}(t)-\bar{u}_{k}$. Combining \eqref{ZHF4.40} and \eqref{ZHF4.41}, we deduce
\begin{align}\label{ZF4.45}
\nonumber&\partial_{t}^{2}Z_{k}-\Delta Z_{k}
+\Gamma_{2}(t)\partial_{t}Z_{k}+\ell(-\Delta)^{-1}Z_{k}
\\\nonumber&\qquad+\Gamma_{1}(t)(\|\partial_{t}u_{k}^{(k)}\|^{2}-\|\partial_{t}\bar{u}_{k}\|^{2})
(\partial_{t}u_{k}^{(k)}+\partial_{t}\bar{u}_{k})
+\mathrm{P}_{k}(g(u_{k}^{(k)})-g(\bar{u}_{k}))
\\=&\tilde{G}_{k}(t):=G_{k}(t)+\ell(-\Delta)^{-1}Z_{k},
\end{align}
where $\Gamma_{1}(t)$, $\Gamma_{2}(t)$ and $G_{k}(t)$ are defined in a manner analogous to that in \eqref{ZHOU4.35}. Using the
multiplier $\partial_{t}Z_{k}+\varepsilon Z_{k}$ in \eqref{ZF4.34}, and following a similar approach to that used in \eqref{ZHOU4.37}, we obtain: 
\begin{align}\label{ZF4.46}
\frac{d}{dt}\mathcal{E}_{Z_{k},\bar{u}_{k}}(t)
+\varepsilon\mathcal{E}_{Z_{k},\bar{u}_{k}}(t)+\sum_{i=1}^{2}\mathcal{Q}^{(i)}_{Z_{k},\bar{u}_{k}}(t)
=\sum_{i=1}^{2}\mathcal{G}^{(i)}_{Z_{k},\bar{u}_{k}}(t)+\tilde{\mathcal{G}}^{(3)}_{Z_{k}}(t),
\end{align}
where $\mathcal{Q}^{(i)}_{Z_{k},\bar{u}_{k}}(t)$ and $\mathcal{G}^{(i)}_{Z_{k},\bar{u}_{k}}(t)$ ($i=1,2$) are defined as in \eqref{ZHOU4.38}, and
\begin{align}\label{ZF4.47}
\tilde{\mathcal{G}}^{(3)}_{Z_{k}}(t)=
2\langle \tilde{G}_{k},\partial_{t}Z_{k}+\varepsilon Z_{k}\rangle.
\end{align}
Arguing as in the derivation of \eqref{ZF4.36}, we can choose
\begin{align}\label{ZHOU4.50}
\varepsilon=\{1,\frac{\lambda_{1}}{4},\frac{J_{0}}{32},\frac{\sqrt[3]{\lambda_{1}J_{0}}}{4},\frac{3\lambda_{1}J_{0}}{16\mathcal{J}^{2}(R_{0}^{2})}\}
\end{align}
sufficiently small to ensure that
\begin{align}\label{ZHOU4.51}
\frac{d}{dt}\mathcal{E}_{Z_{k},\bar{u}_{k}}(t)+\varepsilon\mathcal{E}_{Z_{k},\bar{u}_{k}}(t)
\leq \mathcal{I}_{Z_{k},\bar{u}_{k}}(t),
\end{align}
where
\begin{align*}
\nonumber\mathcal{I}_{Z_{k},\bar{u}_{k}}(t)&=C_{7}\|\tilde{G}_{k}(t)\|^{2}
-\varepsilon C(\kappa_{2},q)\|Z_{k}\|_{L^{q+1}}^{q+1}-\mathcal{Q}^{(4)}_{Z_{k},\bar{u}_{k}}(t)+
\\&\qquad+C_{6}\|\partial_{t}\bar{u}_{k}\|_{L^{\infty}}\langle |Z_{k}|^{2}(1+|\bar{u}_{k}|^{q-2}+|Z_{k}|^{q-2}),1\rangle,
\\\nonumber\mathcal{Q}^{(4)}_{Z_{k},\bar{u}_{k}}(t)&=
\frac{3J_{0}}{4}\|\partial_{t}Z_{k}\|^{2}+\frac{\varepsilon}{4}\|Z_{k}\|_{\mathcal{H}^{1}}^{2}
+\frac{7\varepsilon}{8}\ell\|Z_{k}\|_{\mathcal{H}^{-1}}^{2}
-\varepsilon\mathcal{J}(\|\partial_{t}\bar{u}_{k}\|^{2})\|\partial_{t}Z_{k}\|\|Z_{k}\|
\\&\qquad-4\varepsilon\Gamma_{1}(t)(\|\partial_{t}u_{k}^{(k)}\|^{2}+\|\partial_{t}\bar{u}_{k}\|^{2})\|\partial_{t}Z_{k}\|\|Z_{k}\|.
\end{align*}
Using \eqref{ZHOU4.43} and \eqref{ZHOU4.46}, for sufficiently large $k$, we have
\begin{align*}
\|\partial_{t}\bar{u}_{k}(t)\|^{2}
\leq \mathcal{Q}(\|h\|)~\text{\rm and }~
\|\partial_{t}\bar{u}_{k}(t)\|^{2}+\|\partial_{t}u_{k}^{(k)}\|^{2}\leq M_{1}:=R_{0}^{2}+\mathcal{Q}(\|h\|),~
\forall t\leq T_{1}.
\end{align*}
Unlike in the case of \eqref{ZF4.36}, we cannot reduce the time interval $t\in(-\infty,T_{1}]$. However, because the function $\bar{u}$ is now independent of the parameter $\ell$, we can choose
\begin{align*}
\ell=4\kappa_{1}^{2}+(J_{M_{1}}+J'_{M_{1}}M_{1})^{4}
+(M_{2}+M_{2}M_{1}^{q-2}C^{-1}(\kappa_{2},q)+1)^{2}\varepsilon^{-3},
\end{align*}
where $J_{M_{1}}=\mathcal{J}(M_{1})$, $J'_{M_{1}}=\sup\limits_{s\in[0,M_{1}]}\mathcal{J}'(s)$ and $M_{2}=C_{6}\sqrt{M_{1}}$, such that
\begin{align}\label{ZHOU4.52}
\mathcal{I}_{Z_{k},\bar{u}_{k}}(t)
\leq C_{7}\|\tilde{G}_{k}(t)\|^{2},~~\forall t\leq T_{1}.
\end{align}
Applying Gronwall's inequality to the identity \eqref{ZF4.46}, and using \eqref{ZHOU4.52} along with the fact that $u(t)=\bar{u}(t)$ for $t\leq T_{0}$, we derive estimate
\begin{align}\label{ZHOU4.53}
\mathcal{E}_{Z_{k},\bar{u}_{k}}(t)
\leq
C_{7}\int_{T_{0}}^{t}
e^{-\varepsilon(t-s)}\|\tilde{G}_{k}(s)\|^{2}ds.
\end{align}
Since the term $\ell(-\Delta)^{-1}Z_{k}$ in $\tilde{G}$ converges as
\begin{align*}
\ell(-\Delta)^{-1}Z_{k}\rightarrow \ell(-\Delta)^{-1}Z\text{ strongly in }\mathcal{C}_{loc}((-\infty,T_{1}];L^{2}),
\end{align*}
if follows that
\begin{align}\label{ZHOU4.54}
\tilde{G}_{k}(t)\rightarrow \ell(-\Delta)^{-1}Z
~&\text{strongly in }~\mathcal{C}_{loc}((-\infty,T_{1}];L^{2}).
\end{align}
Taking the limit as $k\rightarrow\infty$ in \eqref{ZHOU4.53}, using \eqref{ZHOU4.50}, \eqref{ZHOU4.54} and
\begin{align*}
\|\tilde{G}_{k}\|_{\mathcal{C}_{loc}((-\infty,T_{1}];L^{2})}\leq C
\end{align*}
with $C$ independent of $k$, we derive
\begin{align}\label{ZHOU4.55}
\|\xi_{u}(t)-\xi_{\bar{u}}(t)\|_{\mathscr{E}}^{2}
\leq 2C_{7}\ell^{2}\int_{T_{0}}^{t}e^{-\varepsilon(t-s)}\|(-\Delta)^{-1}(u(s)-\bar{u}(s))\|^{2}ds,~\forall t\in[T_{0},T_{1}].
\end{align}
Applying Gronwall's inequality to \eqref{ZHOU4.55} and noting that $u(T_{0})=\bar{u}(T_{0})$ we conclude that $u(t)=\bar{u}(t)$ on any interval $[T_{0},T_{1}]$, thereby completing the proof.
\end{proof}
\begin{remark}\label{ZFR4.4}

The proof of Theorem \ref{ZFT6} shows that for any $\xi_{u}\in\bar{\mathfrak{E}}((-\infty,\infty))$, $\xi_{u}$ is the S--S solution of Eq. \eqref{1.1}, which implies that $\bar{\mathfrak{E}}((-\infty,\infty))=\mathfrak{E}((-\infty,\infty))$. Furthermore, we have $
\xi_{u}(t)\in\mathscr{E}^{1} \text{~for~all~}t\in\mathbb{R}$. However, the boundedness of $\xi_u(t)$ in the $\mathscr{E}^1$--norm has not yet been established, and consequently, we cannot directly ascertain the strong attractor $\mathscr{A}_{s}$.
\end{remark}
\subsection{Asymptotic compactness}
\begin{lemma}\label{ZFL4.5}
Assume that the Assumption \ref{A1.1} be in force and assume further $\mathcal{J}(0)>0$, then the semigroup $(S(t),\mathscr{E})$ given by \eqref{Z3.2} associated with Eq. \eqref{1.1} is asymptotically compact, that is for every sequence $\left\{\xi_n\right\}_{n=1}^{\infty} \subset \mathcal{X}$, and every sequence of times $t_n \rightarrow \infty$, there exists a subsequence $n_{k}$ such that
\begin{align}
S\left(t_{n_k}\right) \xi_{n_k}\rightarrow \xi \text { strongly in } \mathscr{E}.
\end{align}
\end{lemma}
\begin{proof}
Let us denote $\xi_{u_{n}}(t)=S(t+t_{n})\xi_{n}$ the corresponding S--S solutions with $t_n \rightarrow \infty$, then $u_{n}$ solves
\begin{align}\label{ZHOU4.57}
\partial_{t}^{2}u_{n}-\Delta u_{n}+\mathcal{J}(\|\partial_{t}u_{n}(t)\|^{2})\partial_{t}u_{n}+g(u_{n})=h,
~t\geq-t_{n}\text{ and }\xi_{u_{n}}(-t_{n})=\xi_{n}\in\mathcal{X}.
\end{align}
We recall that $\xi_{u_{n}}$ is uniformly bounded in $\mathcal{C}([-t_{n},\infty),\mathscr{E})$, then we get that
\begin{align}\label{ZHOU4.58}
\xi_{u_{n}}\rightharpoonup \xi_{u}, \quad\text{in~}\mathcal{C}_{loc}(\mathbb{R},\mathscr{E}_{w})
\end{align}
and $\xi_{u}\in\bar{\mathfrak{E}}((-\infty,\infty))=\mathfrak{E}((-\infty,\infty))$ and $\xi_{u}$ is the S--S solution of Eq. \eqref{1.1} by recalling Theorem \ref{ZFT6} or Remark \ref{ZFR4.4}. In addition, we also know that $\xi_{u_{n}}(0)\rightharpoonup\xi_{u}(0)$ weakly in $\mathscr{E}$. Taking the $L^{2}$--inner product between \eqref{ZHOU4.57} and $\partial_{t}u_{n}+\varrho u_{n}$ ($0<\varrho\ll1$), we derive the following energy type identity
\begin{align}\label{ZHOU4.59}
\frac{d}{dt}\mathcal{E}_{u_{n}}^\varrho(t)+\frac{\varrho}{4}\mathcal{E}_{u_{n}}^\varrho(t)
+\mathcal{Q}_{u_{n}}^\varrho(t)+\mathcal{G}_{u_{n}}^\varrho(t)+\mathcal{I}_{u_{n}}^\varrho(t)=0,
\end{align}
where
\begin{align*}
\mathcal{E}_{u_{n}}^\varrho(t)&=
\|\xi_{u_{n}}\|_{\mathscr{E}}^{2}+2\langle G(u_{n}),1\rangle
+\varrho\langle\partial_{t}u_{n},u_{n}\rangle-2\langle h,u_{n}\rangle,
\nonumber\\\mathcal{Q}_{u_{n}}^\varrho(t)&=\big[2\mathcal{J}(\|\partial_{t}u_{n}\|^{2})
-\frac{5\varrho}{4}\big]\|\partial_{t}u_{n}\|^{2}
+\frac{3\varrho}{4}\|u_{n}\|_{\mathcal{H}^{1}}^{2}
\\\nonumber&\quad+\varrho\mathcal{J}(\|\partial_{t}u_{n}\|^{2})\langle\partial_{t}u_{n},u_{n}\rangle
-\frac{\varrho^{2}}{4}\langle\partial_{t}u_{n},u_{n}\rangle,
\\\mathcal{G}_{u_{n}}^\varrho(t)&=\varrho\langle g(u_{n}),u_{n}\rangle-\frac{\varrho}{2}\langle G(u_{n}),1\rangle,
\quad\mathcal{I}_{u_{n}}^\varrho(t)=-\frac{\varrho}{2}\langle h,u_{n}\rangle.
\end{align*}
Now, integrating Eq. \eqref{ZHOU4.59} with respect to $t\in[-t_{n},0]$, to deduce that
\begin{align}\label{ZHOU4.60}
\mathcal{E}_{u_{n}}^\varrho(0)+\int_{-t_{n}}^{0}e^{\frac{\varrho}{4} s}
\left(\mathcal{Q}_{u_{n}}^\varrho(s)+\mathcal{G}_{u_{n}}^\varrho(s)
+\mathcal{I}_{u_{n}}^\varrho(s)\right)ds
=\mathcal{E}_{u_{n}}^\varrho(-t_{n})e^{-\frac{\varrho}{4}t_{n}}.
\end{align}
In order to pass the limit $n \rightarrow \infty$, we deal with the terms in \eqref{ZHOU4.60} one by one.

Firstly, recalling \eqref{1.5}, we observe that
\begin{align}\label{ZHOU4.61}
\mathcal{G}_{u_{n}}^\varrho(t)\geq C(\kappa_{3},\kappa_{5},|\Omega|):=-7(\kappa_{5}+\kappa_{3})|\Omega|.
\end{align}
Applying the compact embedding $\mathcal{C}_{loc}((-\infty,0];\mathscr{E})\Subset\mathcal{C}_{loc}((-\infty,0];L^{2})$, we obtain
\begin{align}\label{ZHOU4.62}
u_{n}\rightarrow u \text{ strongly in } \mathcal{C}_{loc}((-\infty,0];L^{2}),
\text{ including }u_{n}\rightarrow u,~a.e.
\end{align}
By Fatou's lemma, taking the limit as $n\rightarrow \infty$, we have
\begin{align}\label{ZHOU4.63}
\liminf\limits_{n\rightarrow\infty}\int_{-t_{n}}^{0}e^{\frac{\varrho}{4} s}\mathcal{G}_{u_{n}}^\varrho(s)ds
\geq\int_{-\infty}^{0}e^{\frac{\varrho}{4} s}\mathcal{G}_{u}^\varrho(s)ds.
\end{align}

Secondly, combining \eqref{1.5} with \eqref{ZHOU4.62}, and applying Fatou's lemma alongside the weak lower semicontinuity of the norm, we derive:
\begin{align}\label{ZHOU4.64}
\liminf\limits_{n\rightarrow\infty}\mathcal{E}_{u_{n}}^\varrho(0)\geq\mathcal{E}_{u}^\varrho(0),\quad
\liminf\limits_{n\rightarrow\infty}\int_{t_{n}}^{0}e^{\frac{\varrho}{4} s}\mathcal{I}_{u_{n}}^\varrho(s)ds
\geq\int_{-\infty}^{0}e^{\frac{\varrho}{4} s}\mathcal{I}_{u}^\varrho(s)ds.
\end{align}

Finally, we deal with the remainder term $\mathcal{Q}_{u_{n}}^\varrho$. Denote
\begin{align}\label{ZHOU4.65}
\mathcal{Q}_{u_{n}}^\varrho(t)
=\mathcal{Q}_{u}^\varrho(t)+
\mathcal{R}_{u_{n}}^\varrho(t)+\mathcal{P}_{u_{n}}^\varrho(t),
\end{align}
where
\begin{align*}
\mathcal{R}_{u_{n}}^\varrho(t)
=&2\left(\mathcal{J}(\|\partial_{t}u_{n}\|^{2})\|\partial_{t}u_{n}\|^{2}
-\mathcal{J}(\|\partial_{t}u\|^{2})\|\partial_{t}u\|^{2}\right)
\\&\qquad+\frac{3\varrho}{4}\left(\|u_{n}\|_{\mathcal{H}^{1}}^{2}
-\|u\|_{\mathcal{H}^{1}}^{2}\right)
+\frac{\varrho^{2}}{4}\left(\langle\partial_{t}u,u\rangle-\langle\partial_{t}u_{n},u_{n}\rangle\right),
\end{align*}
and
\begin{align*}
\mathcal{P}_{u_{n}}^\varrho(t)=&\varrho\left(\mathcal{J}(\|\partial_{t}u_{n}\|^{2})\langle\partial_{t}u_{n},u_{n}\rangle
-\mathcal{J}(\|\partial_{t}u\|^{2})\langle\partial_{t}u,u\rangle\right)
\\&\qquad+\frac{5\rho}{4}(\|\partial_{t}u\|^{2}-\|\partial_{t}u_{n}\|^{2}).
\end{align*}
By applying Young's inequality, we obtain
\begin{align}\label{ZHOU4.66}
\nonumber&\varrho\left(\mathcal{J}(\|\partial_{t}u_{n}\|^{2})\langle\partial_{t}u_{n},u_{n}\rangle
-\mathcal{J}(\|\partial_{t}u\|^{2})\langle\partial_{t}u,u\rangle\right)
\\\leq&\varrho C_{\lambda_{1}\delta}\mathcal{J}^{2}(R_{0}^{2})
\left(\|\partial_{t}u_{n}\|^{2}+\|\partial_{t}u\|^{2}\right)
+\varrho\delta\left(\|u_{n}\|_{\mathcal{H}^{1}}^{2}+\|u\|_{\mathcal{H}^{1}}^{2}\right),
\end{align}
and using \eqref{Z3.4} and \eqref{Z3.11}, we find
\begin{align}\label{ZHOU4.67}
\int_{-t_{n}}^{0}e^{\frac{\varrho}{4} s}
\left|\mathcal{P}_{u_{n}}^\varrho(s)\right|ds
\leq\varrho C_{8}\mathcal{Q}(\|h\|^{2})
+8R_{0}^{2}\delta, \quad \forall \delta>0,
\end{align}
here $C_{8}=\frac{5}{2}+C_{\lambda_{1}\delta}\mathcal{J}^{2}(R_{0}^{2})$. Combining \eqref{ZHOU4.65} with \eqref{ZHOU4.67} and utilizing the weak lower semicontinuity of the norm, we obtain
\begin{align}\label{ZHOU4.68}
\liminf\limits_{n\rightarrow\infty}
\int_{-t_{n}}^{0}e^{\frac{\varrho}{4} s}
\mathcal{Q}_{u_{n}}^\varrho(t)ds
\geq\int_{-\infty}^{0}e^{\frac{\varrho}{4} s}\mathcal{Q}_{u}^\varrho(s)ds
-\varrho C_{8}\mathcal{Q}(\|h\|^{2})
-8R_{0}^{2}\delta.
\end{align}

On the other hand, according to Theorem \ref{ZFT6}, $u$ is the S--S solution of problem \eqref{1.1} with enhanced regularity in $\mathscr{E}^{1}$, and clearly, $u$ satisfies the energy equality. By replicating the derivation of \eqref{ZHOU4.60} for the solution $u$, we obtain the energy equality:
\begin{align}\label{ZHOU4.69}
\mathcal{E}_{u}^\varrho(0)+\int_{-\infty}^{0}e^{\frac{\varrho}{4} s}
\left(\mathcal{Q}_{u}^\varrho(s)+\mathcal{G}_{u}^\varrho(s)
+\mathcal{I}_{u}^\varrho(s)\right)ds
=0.
\end{align}
Returning now to \eqref{ZHOU4.60}, and taking the limit as $n \rightarrow \infty$ in equality \eqref{ZHOU4.60}, we use \eqref{ZHOU4.63}, \eqref{ZHOU4.64}, \eqref{ZHOU4.68} and \eqref{ZHOU4.69} to deduce:
\begin{align*}
\nonumber0&\geq\liminf\limits_{n\rightarrow\infty}\left(
\mathcal{E}_{u_{n}}^\varrho(0)+\int_{-t_{n}}^{0}e^{\frac{\varrho}{4} s}
\left(\mathcal{Q}_{u_{n}}^\varrho(s)+\mathcal{G}_{u_{n}}^\varrho(s)
+\mathcal{I}_{u_{n}}^\varrho(s)\right)ds\right)
\\\nonumber&\geq\liminf\limits_{n\rightarrow\infty}
\mathcal{E}_{u_{n}}^\varrho(0)
+\int_{-\infty}^{0}e^{\frac{\varrho}{4} s}
\left(\mathcal{Q}_{u}^\varrho(s)+\mathcal{G}_{u}^\varrho(s)
+\mathcal{I}_{u}^\varrho(s)\right)ds
-\varrho C_{8}\mathcal{Q}(\|h\|^{2})
-8R_{0}^{2}\delta
\\&\geq\liminf\limits_{n\rightarrow\infty}
\mathcal{E}_{u_{n}}^\varrho(0)-\mathcal{E}_{u}^\varrho(0)
-\varrho C_{8}\mathcal{Q}(\|h\|^{2})
-8R_{0}^{2}\delta.
\end{align*}
Thus, we conclude:
\begin{align*}
\mathcal{E}_{u}^\varrho(0)
\leq \liminf\limits_{n\rightarrow\infty}
\mathcal{E}_{u_{n}}^\varrho(0)
\leq \mathcal{E}_{u}^\varrho(0)+\varrho C_{8}\mathcal{Q}(\|h\|^{2})
+8R_{0}^{2}\delta.
\end{align*}
Taking the limit as $\varrho\rightarrow0$, we obtain
\begin{align*}
\mathcal{E}_{u}(0)
\leq \liminf\limits_{n\rightarrow\infty}
\mathcal{E}_{u_{n}}(0)
\leq \mathcal{E}_{u}(0)+8R_{0}^{2}\delta
\end{align*}
for any $\delta>0$, where $\mathcal{E}_{u}(t)=
\|\xi_{u}\|_{\mathscr{E}}^{2}+2\langle G(u),1\rangle-2\langle h,u\rangle$. We take the limit as $\delta\rightarrow0$ to derive:
\begin{align}\label{ZHOU4.70}
\mathcal{E}_{u}(0)
\leq \liminf\limits_{n\rightarrow\infty}
\mathcal{E}_{u_{n}}(0)
\leq \mathcal{E}_{u}(0).
\end{align}
Applying Fatou lemma and weak lower semi-continuous of the norm again, we find that
\begin{align}\label{ZHOU4.71}
\liminf_{n\rightarrow\infty}\langle G(u_{n}(0)),1\rangle\geq\langle G(u(0)),1\rangle,\quad
\liminf_{n\rightarrow\infty}\|\xi_{u_{n}}(0)\|_{\mathscr{E}}^{2}\geq\|\xi_{u}(0)\|_{\mathscr{E}}^{2}.
\end{align}
The equality in \eqref{ZHOU4.70} holds only if inequalities \eqref{ZHOU4.71} are also equalities. Recalling $\xi_{u_{n}}(0)\rightharpoonup\xi_{u}(0)$, we may assume without loss of generality that
\begin{align*}
S(t_{n})\xi_{n}=\xi_{u_{n}}(0)\rightarrow\xi_{u}(0)
\end{align*}
strongly in $\mathscr{E}$. Thus, the asymptotic compactness of the semigroup $S(t)$ is established, completing the proof of the theorem.
\end{proof}
\begin{theorem}\label{ZFT7}
Let the assumptions of Theorem \ref{ZFT5} be satisfied. Then, the solution semigroup $(S(t),\mathscr{E})$ generated by S--S solutions of Eq. \eqref{1.1} possesses a strong global attractor $\mathscr{A}_{s}\subset\mathscr{E}^{1}$. Moreover, we have
\begin{align}\label{ZHOU4.72}
\mathscr{A}_{s}=\mathscr{A}_{w}=\{\xi_{u_{0}}:\xi_{u_{0}}=\xi_{u}(0)\text{~for~some~}\xi_{u}\in\mathfrak{E}((-\infty,\infty))\},
\end{align}
where $\mathscr{A}_{w}$ denotes the weak attractor as defined in Theorem \ref{ZFT3}.
\end{theorem}
\begin{proof}
By Theorem \ref{ZFT2} and Lemma \ref{ZFL4.5}, the semigroup $S(t)$ is dissipative and asymptotically compact. Using \eqref{Z3.3}, we establish that $S(t)$ is continuous in $\mathscr{E}$. Consequently, by applying the abstract attractor existence theorem (refer to \cite{chueshov1, chueshov2}), it follows that $\left(S(t),\mathscr{E}\right)$ possesses a global attractor $\mathscr{A}_{s}$.  Consequently, it follows that $\mathscr{A}_{s}\subset\{\xi_{u_{0}}:\xi_{u_{0}}=\xi_{u}(0)\text{~for~some~}\xi_{u}\in\mathfrak{E}((-\infty,\infty))\}$. On the other hand, by applying Theorem \ref{ZFT6} and Remark \ref{ZFR4.4}, we observe that $\mathfrak{E}((-\infty,\infty))=\bar{\mathfrak{E}}((-\infty,\infty))$ consists of smooth solutions which are the S--S ones. Thus, we obtain $\mathscr{A}_{s}\subset\mathscr{E}^{1}$ and equality \eqref{ZHOU4.72}.
\end{proof}
\begin{remark}
\textbf{(Characterization)}
Let us define the functional $\Phi(\xi_{u}): \mathscr{E}\rightarrow\mathbb{R}$ as $\xi_{u}\mapsto\Phi(\xi_{u})$,
where
$$
\Phi(\xi_{u}):=\mathcal{E}_{u}=\|\xi_{u}\|_{\mathscr{E}}^{2}+2\langle G(u),1\rangle-2\langle h,u\rangle
$$.
It follows directly from \eqref{Zfeng3.1} that the function $t\mapsto\Phi(S(t)\xi_{u_{0}})$ is non-increasing for every $\xi_{u_{0}}\in\mathscr{E}$. Rewriting equation \eqref{Zfeng3.1} yields
\begin{align}\label{Zfeng4.73}
\Phi(S(t)\xi_{u_{0}})+2\int_{0}^{t}\mathcal{J}(\|\partial_{t}u(s)\|^{2})\|\partial_{t}u(s)\|^{2}ds=\Phi(\xi_{u_{0}}),\quad t>0,
\end{align}
for every $\xi_{u_{0}}\in\mathscr{E}$. From this, we can easily deduce that
$$\Phi(S(t)\xi_{u_{0}})=\Phi(\xi_{u_{0}})\Leftrightarrow \xi_{u_{0}}\in\mathscr{N}
, \quad t>0,
$$
where $\mathscr{N}=\{\xi_{u}\in\mathscr{E}:S(t)\xi_{u}=\xi_{u},~\text{for all }t\geq0\}$ denotes the set of stationary points of the dynamical system $(S(t),\mathscr{E})$. Consequently, we have $\mathscr{A}_{s}=\mathscr{M}^{u}(\mathscr{N})$ and the global attractor $\mathscr{A}_{s}$ consists of full trajectories $\Xi=\{\xi_{u}(t): t\in\mathbb{R}\}$ that satisfy
\begin{align*}
\lim _{t \rightarrow\pm\infty} \operatorname{dist}_\mathscr{E}(\xi_{u}(t), \mathscr{N})=0
\end{align*}
as established by \cite[Theorem 2.4.5]{chueshov1}. Here, $\mathscr{M}^{u}(\mathscr{N})$ represents the unstable manifold (see \cite[Definition 2.3.10]{chueshov1}).
\end{remark}

\section{Dynamics of strong solutions}
\subsection{Dissipativity}
In this subsection,  we aim to establish the dissipativity of the solution semigroup $S(t)$ in $\mathscr{E}^{1}$. For any $\delta>0$, let us denote the $\delta$--neighborhood of $\mathscr{A}_{s}$ in $\mathscr{E}$ by
\begin{align}\label{ZHOUF5.1}
\mathcal{B}_\delta:=\left\{\xi \in \mathscr{E}: \operatorname{dist}_{\mathscr{E}}\left(\xi, \mathscr{A}_{s}\right) \leq \delta\right\},
\end{align}
where $\mathscr{A}_{s}$ is the strong global attractor of $S(t)$ established in Theorem \eqref{ZFT7}. Clearly, $\mathcal{B}_\delta$ is a bounded absorbing set for $(S(t),\mathscr{E})$ for any $\delta>0$.

\begin{lemma}\label{ZFL5.1}
Choosing $\delta_0>0$ small enough, then there exist a time $T_{0}:=T\left(\mathscr{A}_{s}\right)>0$ and a constant $C_{9}=C_{9}\left(\mathscr{A}_s\right)>0$ such that
\begin{align*}
\|u\|_{L^4\left([0,T]; L^{12}\right)} \leq C_{9}
\end{align*}
for any S--S solution $u(t)$ with initial data $\xi_{u}(0) \in \mathcal{B}_{\delta_0}$.
\end{lemma}
The proof of this lemma can be treated by repeating verbatim the arguments of \cite[Lemma 4.2]{zhongyt} and for this reason is omitted.
\begin{theorem}\label{ZFT8}
Assume that the condition in Assumption \eqref{A1.1} is satisfied, and in addition, that $\mathcal{J}(0)>0$. Then the semigroup $\left(S(t),\mathscr{E}^{1}\right)$ is dissipative.
\end{theorem}
\begin{proof}
We divide the proof into several steps.

\textbf{Step 1.} Let $B$ be an arbitrary bounded set in $\mathscr{E}^{1}$, then there exists a time $t_{1}=t_{1}(B)>0$ such that
\begin{align*}
S(t)B\subset\mathcal{B}_{\delta_{0}},\quad t\geq t_{1},
\end{align*}
$\mathcal{B}_{\delta_{0}}$ defined as above by \eqref{ZHOUF5.1}. Define
\begin{align*}
\tilde{\mathcal{B}}:=\overline{\bigcup\limits_{t \geq t_{1}} S(t) B}^{\mathscr{E}}.
\end{align*}
Consequently, $\tilde{\mathcal{B}}\subset\mathcal{B}_{\delta_{0}}$ is a compact (in $\mathscr{E}$) positively invariant absorbing set for $(S(t),\mathscr{E}^{1})$.

\textbf{Step 2.} Let $\xi_{u}(0)\in\tilde{\mathcal{B}}$ and let
\begin{align*}
\mathscr{K}:=\{u(\cdot)|_{[0,\infty)}: \xi_{u}
\text{ is the S--S solution with initial data } \xi_{u}(0)\in\tilde{\mathcal{B}}\}.
\end{align*}
Obviously, $\mathscr{K}$ is positively invariant under the translation: $T_{h}\mathscr{K}\subset\mathscr{K}$, $\forall h\geq0$, where $(T _{h}u)(\cdot):=u(\cdot+h)$. Denote the restriction of the trajectory in $\mathscr{K}$ to the time interval $t\in[0,1]$ as $
\mathcal{D}:=\{u(\cdot)|_{[0,1]},u\in\mathscr{K}\}$.

Claim \#1: $\mathcal{D}$ is a compact set of $L^{4}([0,1];L^{12})$, i.e.,
\begin{align}\label{ZHOUF5.2}
\mathcal{D}\Subset L^{4}([0,1];L^{12}).
\end{align}
\textit{proof of claim}: Applying Lemma \ref{ZFL5.1}, recalling $\tilde{\mathcal{B}}\subset\mathcal{B}_{\delta_{0}}$ and $T_{h}\mathscr{K}\subset\mathscr{K}$, we deduce that
\begin{align}\label{ZHOU5.2}
\sup\limits_{t\geq0}\|u\|_{L^{4}([t,t+1];L^{12})}\leq\frac{C_{9}}{\min\{T_{0},1\}}, \quad u\in\mathscr{K}.
\end{align}
Define a map $S^{1}$: $\tilde{\mathcal{B}}\rightarrow L^{4}([0,1];L^{12})$ by
\begin{align*}
S^{1}:\xi_{u}(0)\rightarrow u(\cdot)|_{[0,1]}.
\end{align*}
Let $\xi_{u_{i}}$ be the S--S solution of Eq. \eqref{1.1} with initial value $\xi_{u_{i}}(0)\in\tilde{\mathcal{B}}$, $i=1,2$ and let $w=u_{1}-u_{2}$ and $\widetilde{w}=u_{1}+u_{2}$, then we get
\begin{align}\label{ZHOU5.3}
\partial_{t}^{2}w-\Delta w
+\Gamma_{1}(t)(\|\partial_{t}u_{1}\|^{2}-\|\partial_{t}u_{2}\|^{2})\partial_{t}\widetilde{w}
+\Gamma_{2}(t)\partial_{t}w
+g(u_{1})-g(u_{2})=0
\end{align}
for $\Gamma_{1}(t)$ and $\Gamma_{2}(t)$ have the same form as that in \eqref{ZHOU4.35}. Using \eqref{ZHOU5.2}, the inequality \eqref{Z3.3} can be improved as
\begin{align}\label{ZHOU5.4}
\|\xi_{u_{1}}(t)-\xi_{u_{2}}(t)\|_{\mathscr{E}}
\leq Ce^{C_{10}t}\|\xi_{u_{1}}(0)-\xi_{u_{2}}(0)\|_{\mathscr{E}},\quad \forall t\in[0,1],
\end{align}
where $C$ and $C_{10}$ are independent of $\xi_{u_{i}}$ $i=1,2$. Combining \eqref{ZHOU5.3} and \eqref{ZHOU5.4}, we conclude
\begin{align}\label{ZHOU5.5}
\nonumber&\|g(u_{1})-g(u_{2})\|_{L^{1}(0,1;L^{2})}
\\\nonumber\leq &C\int_{0}^{1}(1+\|u_{1}(t)\|_{L^{12}}^{4}+\|u_{2}(t)\|_{L^{12}}^{4})\|w(t)\|_{L^{6}}dt
\\\nonumber\leq &C(1+\|u_{1}\|_{L^{4}(0,1;L^{12})}+\|u_{2}\|_{L^{4}(0,1;L^{12})})\|\xi_{w}\|_{\mathcal{C}([0,1];\mathscr{E})}
\\\leq &C\|\xi_{u_{1}}(0)-\xi_{u_{2}}(0)\|_{\mathscr{E}}.
\end{align}
By applying Lemma \ref{SL} to \eqref{ZHOU5.3} and utilizing the result from \eqref{ZHOU5.5}, we obtain, after some calculations, that
\begin{align*}
\|u_{1}-u_{2}\|_{L^{4}([0,1];L^{12})}
\leq C \|\xi_{u_{1}}(0)-\xi_{u_{2}}(0)\|_{\mathscr{E}},
\end{align*}
where the constant $C$ is independent of $\xi_{u_{i}}(0)\in\tilde{\mathcal{B}}$, $i=1,2$. Consequently, the map $S^{1}$ is continuous on $\tilde{\mathcal{B}}$. Since $\tilde{\mathcal{B}}$ is compact in $\mathscr{E}$, the result follows.

\textbf{Step 3.} Combining \eqref{ZHOUF5.2} and \eqref{ZHOU5.2}, for any $\varepsilon>0$, we can decompose the solution $u\in\mathscr{K}$ into two parts $u=\hat{u}+\tilde{u}$, where
\begin{align}\label{ZHOU5.7}
\sup_{t\geq0}\|\tilde{u}\|_{L^{4}([t,t+1];L^{12})}\leq\varepsilon\quad\text{and}\quad
\|\hat{u}(t)\|_{\mathcal{C}([0,+\infty);\mathcal{H}^{2})}\leq C_{\varepsilon},
\end{align}
where the constant $C_{\varepsilon}$ depends on $\varepsilon$ and $\mathscr{A}_{s}$, but independent of $u$. The subsequent estimates will be derived through a formal argument, which can be rigorously justified using the Faedo--Galerkin method.

Differentiate Eq. \eqref{1.1} with respect to $t$ and denoting $\theta:=\partial_{t}u$, we obtain
\begin{align}\label{ZHOU5.8}
\partial_{t}^{2}\theta-\Delta \theta
+\mathcal{J}(\|\partial_{t}u(t)\|^{2})\partial_{t}\theta
+2\mathcal{J}'(\|\partial_{t}u(t)\|^{2})\langle \theta,\partial_{t}\theta\rangle \theta
+g'(u)\theta=0
\end{align}
with the initial condition
\begin{align}\label{ZHOU5.9}
\xi_{\theta}(0)=\left(\partial_{t}u(0),\partial_{t}^{2}u(0))=(u_{1},\Delta u_{0}-g(u_{0})-\mathcal{J}(\|u_{1}\|^{2})u_{1}+h\right)\in\mathscr{E}.
\end{align}
Taking the multiplier $\partial_{t}\theta+\alpha\theta$ in \eqref{ZHOU5.8}, we can discover
\begin{align}\label{ZHOU5.10}
\frac{d}{dt}\mathcal{E}_{\theta}(t)
+\mathcal{Q}_{\theta}(t)+\mathcal{G}_{\theta}(t)=0,
\end{align}
where
\begin{align*}
\mathcal{E}_{\theta}(t)
&=\|\xi_{\theta}\|_{\mathscr{E}}^{2}+2\alpha\langle\partial_{t}\theta,\theta\rangle,
\\\mathcal{Q}_{\theta}(t)
&=2[\mathcal{J}(\|\partial_{t}u\|^{2})-\alpha]\|\partial_{t}\theta\|^{2}
+2\alpha\|\theta\|_{\mathcal{H}^{1}}^{2}
+4\mathcal{J}'(\|\partial_{t}u\|^{2})\langle\theta,\partial_{t}\theta\rangle^{2}
\\&\quad+2\alpha\mathcal{J}(\|\partial_{t}u\|^{2})\langle\partial_{t}\theta, \theta\rangle
+4\alpha\mathcal{J}'(\|\partial_{t}u\|^{2})\langle\partial_{t}\theta, \theta\rangle\|\theta\|^{2},
\\\mathcal{G}_{\theta}(t)
&=2\langle g'(u)\theta,\partial_{t}\theta\rangle+2\alpha\langle g'(u),\theta^{2}\rangle.
\end{align*}
Choosing
\begin{align*}
\alpha=\min\{\frac{\sqrt{\lambda_{1}}}{2},\frac{J_{0}}{8},\frac{2J_{0}}{C_{11}^{2}}\}
\end{align*}
with $C_{11}=\frac{2\mathcal{J}(R_{0}^{2})+4M_{0}R_{0}^{2}}{\sqrt{\lambda_{1}}}$ small enough to deduce that
\begin{align}\label{ZHOUF5.11}
\frac{1}{2}\|\xi_{\theta}(t)\|_{\mathcal{E}}^{2}
\leq \mathcal{E}_{\theta}(t)
\leq \frac{3}{2}\|\xi_{\theta}(t)\|_{\mathcal{E}}^{2}
\end{align}
and
\begin{align}\label{ZHOU5.11}
\frac{d}{dt}\mathcal{E}_{\theta}(t)+\frac{2\alpha}{3}\mathcal{E}_{\theta}(t)
\leq2\beta_{\varepsilon}(t)\mathcal{E}_{\theta}(t)+C_{\alpha,\kappa_{1},\mathscr{A}_{s}}\|\partial_{t}u\|^{2}-2\langle g'(u)\theta,\partial_{t}\theta\rangle.
\end{align}
Using the decomposition given in \eqref{ZHOU5.7}, we obtain the following estimate:
\begin{align}\label{ZHOU5.12}
\begin{split}
&|\langle g'(u)\theta,\partial_{t}\theta\rangle|
\\\leq &|\langle (g'(\tilde{u}+\hat{u})-g'(\hat{u}))\theta,\partial_{t}\theta\rangle|+
|\langle g'(\hat{u})\theta,\partial_{t}\theta\rangle|
\\\leq &C\langle(1+|\hat{u}|^{3}+|\tilde{u}|^{3})|\tilde{u}|,|\theta||\partial_{t}\theta|\rangle
+\|g'(\hat{u})\|_{L^{\infty}}\|\theta\|\|\partial_{t}\theta\|
\\\leq&C(1+\|\tilde{u}\|_{L^{12}}^{3}+\|\hat{u}\|_{L^{12}}^{3})\|\tilde{u}\|_{L^{12}}
\|\xi_{\theta}\|_{\mathscr{E}}^{2}+
\frac{\alpha}{12}\|\partial_{t}\theta\|^{2}+C_{\alpha,\mathscr{A}_{s}}\|\partial_{t}u\|^{2}
\\\leq&\beta_{\varepsilon}(t)\|\xi_{\theta}\|_{\mathscr{E}}^{2}+C_{\alpha,\mathscr{A}_{s}}\|\partial_{t}u\|^{2}
+\frac{\alpha}{12}\|\partial_{t}\theta\|^{2}
\end{split}
\end{align}
with $\beta_{\varepsilon}(t)=C(1+\|\tilde{u}\|_{L^{12}}^{3}+\|\hat{u}\|_{L^{12}}^{3})\|\tilde{u}\|_{L^{12}}$. Owing to \eqref{ZHOU5.2} and \eqref{ZHOU5.7}, we conclude
\begin{align}\label{ZHOU5.13}
\nonumber&\int_{t}^{t+1}\beta_{\varepsilon}(r)dr
\\\nonumber\leq&C\left(\int_{t}^{t+1}(1+\|\hat{u}\|_{L^{12}}^{3}+\|\tilde{u}\|_{L^{12}}^{3})^{\frac{4}{3}}dr\right)^{\frac{3}{4}}
\left(\int_{t}^{t+1}\|\hat{u}\|_{L^{12}}^{4}dr\right)^{\frac{1}{4}}
\\\leq &C\left(1+\|\tilde{u}\|_{L^{4}(t,t+1;L^{12})}^{3}+\|\hat{u}\|_{L^{4}(t,t+1;L^{12})}^{3}\right)
\|\tilde{u}\|_{L^{4}(t,t+1;L^{12})}
\leq C \varepsilon,\quad\forall t\geq0
\end{align}
for some positive constant $C$ independent of $t$, $u$ and $\varepsilon$. Combining now \eqref{ZHOU5.11}--\eqref{ZHOU5.12} and employing Gronwall's inequality, we deduce
\begin{align}\label{ZHOU5.14}
\nonumber\|\xi_{\theta}(t)\|_{\mathscr{E}}^{2}
\leq& e^{-\int_{0}^{t}(\frac{\alpha}{4}-2\beta_{\varepsilon}(r))dr}\mathcal{Q}(\|\xi_{\theta}(0)\|_{\mathscr{E}}^{2})
\\&\qquad+C\int_{0}^{t}e^{-\int_{0}^{r}(\frac{\alpha}{4}-2\beta_{\varepsilon}(\mu))d\mu}\|\xi_{u}(r)\|_{\mathscr{E}}^{2}dr
,\quad \forall t\geq0
\end{align}
for some monotone function $\mathcal{Q}(\cdot)$ and positive constant $\alpha$ which is independent of $\varepsilon$ and $u$. Selecting $\varepsilon$ sufficiently small and combining \eqref{ZHOUF5.11}, \eqref{ZHOU5.13} and \eqref{ZHOU5.14} to derive that
\begin{align}\label{ZHOU5.15}
\nonumber\|\xi_{\theta}(t)\|_{\mathscr{E}}^{2}
&\leq e^{-\frac{\alpha}{8}t}\mathcal{Q}(\|\xi_{\theta}(0)\|_{\mathscr{E}}^{2})
+C\|\xi_{u}\|_{\mathcal{C}(\mathbb{R}_{+};\mathscr{E})}^{2}
\\&\leq e^{-\frac{\alpha}{8}t}\mathcal{Q}(\|\xi_{\theta}(0)\|_{\mathscr{E}}^{2})
+C(\kappa_{1},\mathscr{A}_{s},R_{0}^{2}),\quad \forall t\geq0
\end{align}
for some monotone function $\mathcal{Q}(\cdot)$. Recalling \eqref{ZHOU5.9}, we see that in fact
\begin{align*}
\|\xi_{\theta}(0)\|_{\mathscr{E}}^{2}
\leq C(\|\xi_{u}(0)\|_{\mathscr{E}^{1}}^{2}+\|h\|^{2}).
\end{align*}
Substituting this estimate into \eqref{ZHOU5.15}, we obtain
\begin{align}\label{ZHOU5.17}
\|\xi_{\theta}(t)\|_{\mathscr{E}}^{2}
\leq e^{-\frac{\alpha}{8}t}\mathcal{Q}(\|\xi_{u}(0)\|_{\mathscr{E}^{1}}^{2}+\|h\|^{2})
+C(\kappa_{1},\mathscr{A}_{s},R_{0}^{2}),\quad \forall t\geq0.
\end{align}
From Eq. \eqref{1.1} and Assumption \ref{A1.1}, we deduce that
\begin{align}\label{ZHOU5.18}
\|u(t)\|_{\mathcal{H}^{2}}^{2}\leq C(\|h\|^{2}+\|\xi_{\theta}(t)\|^{2}),\quad\forall t\geq 0.
\end{align}
Combining \eqref{ZHOU5.17} and \eqref{ZHOU5.18}, we derive the estimate
\begin{align}\label{ZHOU5.19}
\|\xi_{u}(t)\|_{\mathscr{E}^{1}}^{2}
\leq e^{-\frac{\alpha}{8}t}\mathcal{Q}(\|\xi_{u}(0)\|_{\mathscr{E}^{1}}^{2}+\|f\|^{2})
+C(\kappa_{1},\mathscr{A}_{s},R_{0}^{2})
,\quad\forall t\geq0,~\xi_{u}(0)\in\tilde{\mathcal{B}}.
\end{align}
In particular, for any $\xi_{u}(0)\in S(t_{1})B$, the above estimate holds. Then there exists a time $t=t_{2}(B)$ such that
\begin{align}\label{ZHOU5.20}
\|\xi_{u}(t)\|_{\mathscr{E}^{1}}^{2}
\leq
1+C(\kappa_{1},\mathscr{A}_{s},R_{0}^{2})
,\quad\forall t\geq t_{1}+t_{2},~\xi_{u}(0)\in B.
\end{align}
The set
\begin{align}\label{ZHOU5.21}
\mathds{B}_{1}:=\{\xi_{u} \in \mathscr{E}^{1}:\|\xi_{u}\|_{\mathscr{E}^{1}}^{2}\leq R_{1}^{2}:=1+C(\kappa_{1},\mathscr{A}_{s},R_{0}^{2})\}
\end{align}
is a bounded absorbing set for $S(t)$, completing the proof.
\end{proof}
\begin{corollary}\label{ZFC5.3}
Assuming the hypotheses of Theorem \ref{ZFT8} are satisfied, the global attractor $\mathscr{A}_s$ of the solution semigroup $S(t)$ associated with Eq. \eqref{1.1} is a bounded set in $\mathscr{E}^1$.
\end{corollary}

\begin{proof}
Since the global attractor $\mathscr{A}_{s}$ is a compact and invariant set in $\mathscr{E}$, the proof of this corollary follows almost verbatim from the proof of the previous theorem. Consequently, the detailed proof is omitted.
\end{proof}
\subsection{Exponential attractor}
In \cite{carvalho1,carvalho2}, Carvalho and Sonner introduced a novel type of exponential attractor, specifically the time-dependent exponential attractor. Notably, this attractor is periodic and corresponds to the exponential attractors in the discrete case, satisfying the same dimension estimates as those for discrete semigroups.
\begin{definition}\label{ZFD6.1}
A seminorm $\textbf{n}_{X}(\cdot)$ on the Banach space $(X,\|\cdot\|_{X})$ is said to be compact if any bounded sequence $\{x_{m}\}\subset X$ contains a subsequence $\{x_{m_{k}}\}$ such that
\begin{align*}
\textbf{n}_{X}
(x_{m_{k}}-x_{m_{l}})\rightarrow0 \quad\text{as } k,l\rightarrow\infty.
\end{align*}
\end{definition}

Let $\mathcal{B}_{S}$ be a closed, bounded subset of $E$ such that
\begin{align*}
S(t)\mathcal{B}_{S}\subset\mathcal{B}_{S},\quad\forall t\geq0,
\end{align*}
then the triple $\left(S(t), \mathcal{B}_{S}, E\right)$ is referred to as an autonomous dynamical system acting on $\mathcal{B}_{S}$, see \cite{efendiev} for more details. According to Theorem \ref{ZFT8}, we may assume without loss of generality that the absorbing set $\mathds{B}_{1}$ constructed in \eqref{ZHOU5.21} is positively invariant. Thus, $\left(S(t),\mathds{B}_{1}, \mathscr{E}^{1}\right)$ is an autonomous dynamical system.

\begin{definition}
We call the family $\mathcal{M}=\{\mathcal{M}(s)|s\in\mathbb{R}\}$ a time-dependent exponential attractor for the semigroup $\{S(t)\}_{t\geq0}$ on $\mathcal{B}_{S}$ if:
\begin{enumerate}
  \item there exists $0<\varpi<\infty$ such that $\mathcal{M}(s)=\mathcal{M}(\varpi+s)$,\quad$\forall s\in\mathbb{R}$;

   \item the subsets $\mathcal{M}(s)\subset\mathcal{B}_{S}$ are non-empty and compact, $\forall s\in\mathbb{R}$. The fractal dimension of the sets $\mathcal{M}(s)$ is uniformly bounded, i.e.,
\begin{align*}
\sup _{s \in \mathbb{R}}\dim_{\mathscr{F}}^{E}(\mathcal{M}(s))<+\infty
\end{align*}
where $\dim_{\mathscr{F}}^{E}(A)=\limsup\limits _{\epsilon \rightarrow 0} \frac{\ln N(A, \epsilon)}{\ln (1 / \epsilon)}$ and $N(A, \epsilon)$ denotes the cardinality of the minimal covering of the set $A$ by the closed subsets of diameter $\leq 2 \epsilon$;
  \item
      the family is positive semi-variant, that is
      \begin{align*}
      S(t)\mathcal{M}(s)\subset\mathcal{M}(t+s),\quad\forall t\geq0,\quad\forall s\in\mathbb{R};
      \end{align*}
  \item there exist two positive constants $\alpha$ and $\beta$ such that
\begin{align*}
\sup\limits_{s\in[0,\varpi]}dist_{E}(S(t)\mathcal{B}_{S},\mathcal{M}(s))\leq\alpha e^{-\beta t}, \quad \forall t\geq0.
\end{align*}
\end{enumerate}
\end{definition}

We now present a new criterion for the existence of exponential attractors in an autonomous dynamical system $\left(S(t),\mathcal{B}_{S}, E\right)$, the proof can be found in \cite[Theorem 3.8]{zhou}.
\begin{theorem}\label{ZHOUT}
Let $\mathcal{B}_{S}$ be a bounded closed subset of Banach space $E$, and $\left(S(t),\mathcal{B}_{S}, E\right)$ be an autonomous dynamical system. Assume that
\begin{enumerate}
  \item there exist positive constants $T$ and $L_{\scriptscriptstyle T}$ such that
      \begin{align}\label{ZHOU5.22}
      \|S(t)x-S(t)y\|_{E}\leq L_{\scriptscriptstyle T} \|x-y\|_{E},\quad \forall x,y\in\mathcal{B}_{S},~t\in[0,T];
      \end{align}
  \item there exist a positive time $t^{*}$ and a compact seminorm $\textbf{n}_{Z}(\cdot)$ on a Banach space $Z$, and there exists mapping $\mathfrak{C}:\mathcal{B}_{S}\rightarrow Z$ such that
      \begin{align}\label{ZHOU5.23}
      \|\mathfrak{C}x-\mathfrak{C}y\|_{Z}\leq\nu\|x-y\|_{E},\quad \forall x,y\in\mathcal{B}_{S};
      \end{align}
      \begin{align}\label{ZHOU5.24}
      \|S(t^{*})x-S(t^{*})y\|_{E}\leq \eta\|x-y\|_{E}+
      \textbf{n}_{Z}\big(\mathfrak{C}x-\mathfrak{C}y\big),\quad \forall x,y\in\mathcal{B}_{S},
      \end{align}
  where $0\leq\eta<\textstyle\frac{1}{2}$, $\nu>0$ are constants.
\end{enumerate}
Then, for any $\iota\in(0,\frac{1}{2}-\eta)$, the dynamical system $\left(S(t),\mathcal{B}_{S}, E\right)$ possesses a time-dependent exponential attractor $\mathcal{M}=\{\mathcal{M}^{\iota}(t):t\in\mathbb{R}\}$. Moreover, the fractal dimension of its sections can be estimated by
\begin{align*}
\dim_{\mathscr{F}}^{E}(\mathcal{M}^{\iota}(t))
\leq \log_{\frac{1}{2(\iota+\eta)}}
(N_{\frac{\iota}{\nu}}^{\textbf{n}_{Z}}(B_{1}^{Z}(0))),\text{ for all }t\in\mathbb{R},
\end{align*}
where $B_{r}^{Z}(a)$ denotes the ball of radius $r>0$ and center $a\in Z$ in the metric space $Z$, and $N_{\epsilon}^{\textbf{n}_{Z}}(A)$ denotes the minimal number of $\epsilon$--balls with centers in $A$ needed to cover the subset $A\subset Z$ with $seminorm$ $\textbf{n}_{Z}$.
\end{theorem}

\begin{lemma}\label{ZHOUL5.6}
Assume that the conditions in Assumptions \ref{A1.1} are satisfied, and that $\mathcal{J}(0)>0$. Then, for any two solutions $\xi_{u}(t)$ and $\xi_{v}(t)$ with initial data $\xi_{u_{0}}=(u_{0},u_{1})$ and $\xi_{v_{0}}=(v_{0},v_{1})$, respectively, the following Lipschitz continuity holds:
\begin{align}\label{ZHOU5.25}
\|\xi_{u}(t)-\xi_{v}(t)\|_{\mathscr{E}^{1}}
\leq\mathrm{e}^{Lt}\|\xi_{u_{0}}-\xi_{v_{0}}\|_{\mathscr{E}^{1}},
\quad \forall t\geq0,~\xi_{u_{0}},\xi_{v_{0}}\in\mathds{B}_{1},
\end{align}
where $L$ depends on $\mathds{B}_{1}$, but independent of $t$ and the concrete choice of $\xi_{u_{0}}$ and $\xi_{v_{0}}$.
\end{lemma}
\begin{proof}
Since $u(t)$ is bounded in $\mathcal{H}^2$ and $\mathcal{H}^2 \subset \mathcal{C}(\bar{\Omega})$, the argument is analogous to those used in linear cases. The proof of this lemma follows standard techniques and is left to the reader.
\end{proof}

The following theorem can be considered as the one of the  main results of this section.
\begin{theorem}\label{ZFT9}
In addition to Assumption \ref{A1.1}, suppose that $\mathcal{J}(0)>0$. For any $0<\eta<1$, define $t^{*}=\frac{\ln\frac{48}{\eta}}{\gamma_{0}}$, where $\gamma_{0}$ is specified in \eqref{ZHOUfeng5.28}.

Then, for any $\iota\in(0,\frac{1-\eta}{2})$, the semigroup $\left(S(t),\mathscr{E}^{1}\right)$  possesses a time-dependent exponential attractor $\mathfrak{A}=\{\mathscr{A}_{exp}^{\iota}(s):s\in\mathbb{R}\}$ in $\mathscr{E}^{1}$ which satisfies the following properties:
\begin{enumerate}[(i)]
\item There exists a positive constant $\varpi>0$ such that $\mathscr{A}_{exp}^{\iota}(s)=\mathscr{A}_{exp}^{\iota}(\varpi+s)$,\quad$\forall s\in\mathbb{R}$\label{i};
\item The family $\mathfrak{A}=\{\mathscr{A}_{exp}^{\iota}(s):s\in\mathbb{R}\}$ is positive semi-variant, that is\label{ii}
      \begin{align*}
      S(t)\mathscr{A}_{exp}^{\iota}(s)\subset\mathscr{A}_{exp}^{\iota}(t+s),\quad\forall t\geq0,\quad\forall s\in\mathbb{R};
      \end{align*}
\item There exists a positive constant $\beta$ such that, for every bounded subset $\mathcal{B}$ of $\mathscr{E}^{1}$,\label{iii}
\begin{align*}
\sup\limits_{s\in[0,\varpi]}dist_{\mathscr{E}^{1}}(S(t)\mathcal{B},\mathscr{A}_{exp}^{\iota}(s))\leq \mathcal{Q}\big(\|\mathcal{B}\|_{\mathscr{E}^{1}}\big)e^{-\beta t}, \quad \forall t\geq0;
\end{align*}
\item Each $\mathscr{A}_{exp}^{\iota}(s)$ is compact in $\mathscr{E}^{1}$ and its fractal dimension in $\mathscr{E}^{1}$ is uniformly bounded. Specifically,\label{iv}
      \begin{align*}
      \sup_{s\in\mathbb{R}}\dim_{\mathscr{F}}^{\mathscr{E}^{1}}\big(\mathscr{A}_{exp}^{\iota}(s)\big)
      \leq\log_{\frac{1}{2(\iota+\eta)}}
\left(N_{\frac{\iota}{\nu}}^{\textbf{n}_{\mathds{Z}}}\left(B_{1}^{\mathds{Z}}(0)\right)\right),\text{ for all }s\in\mathbb{R},
      \end{align*}
      where $\mathds{Z}$, $\textbf{n}_{\mathds{Z}}$ and $\nu$ are respectively from \eqref{ZHOUFE5.29}, \eqref{ZHOUFE5.31} and \eqref{ZHOUF5.32}.
\end{enumerate}
\end{theorem}
\begin{proof}
Let $\xi_{u}$ and $\xi_{v}$ be the S--S solutions of Eq. \eqref{1.1} with initial value $\xi_{u_{0}},\xi_{v_{0}}\in\mathds{B}_{1}$. Then we have
\begin{align}\label{ZHOU5.26}
\partial_{t}^{2}w-\Delta w
+\Gamma_{1}(t)(\|\partial_{t}u\|^{2}-\|\partial_{t}v\|^{2})
(\partial_{t}u+\partial_{t}v)
+\Gamma_{2}(t)\partial_{t}w
+g(u)-g(v)=0,
\end{align}
where $\Gamma_{1}(t)$ and $\Gamma_{2}(t)$ are defined as in \eqref{ZHOU4.35}. Taking the multiplier $-\Delta(\partial_{t}w+\gamma w)$ in \eqref{ZHOU5.26}, to deduce
\begin{align}\label{ZHOU5.27}
\frac{d}{dt}\mathcal{E}_{w}(t)
+\mathcal{Q}_{w}(t)+\mathcal{J}_{w}(t)+\mathcal{G}_{w}(t)=0,
\end{align}
where
\begin{align*}
\mathcal{E}_{w}(t)&=\|\xi_{w}\|_{\mathscr{E}^{1}}^{2}
+2\gamma\langle\langle\partial_{t}w,w\rangle\rangle,
\\\mathcal{Q}_{w}(t)&=2(\Gamma_{2}(t)-\gamma)\|\partial_{t}w\|_{\mathcal{H}^{1}}^{2}
+2\gamma\|w\|_{\mathcal{H}^{2}}^{2},
\\\mathcal{J}_{w}(t)&=2\Gamma_{1}(t)(\|\partial_{t}u\|^{2}-\|\partial_{t}v\|^{2})
\langle\langle\partial_{t}u+\partial_{t}v,\partial_{t}w+\gamma w\rangle\rangle
+2\gamma\Gamma_{2}(t)\langle\langle\partial_{t}w,w\rangle\rangle,
\\\mathcal{G}_{w}(t)&=2\langle g'(u)\nabla u-g'(v)\nabla v,\nabla\partial_{t}w+\gamma\nabla w\rangle.
\end{align*}
Choosing
\begin{align}\label{ZHOUfeng5.28}
\gamma_{0}=\{1,\frac{\sqrt{\lambda_{1}}}{2},\frac{J_{0}}{2}\}
\end{align}
small enough and using Gronwall's inequality to \eqref{ZHOU5.27}, to discover
\begin{align*}
\|\xi_{w}(t)\|_{\mathscr{E}^{1}}^{2}
\leq 3e^{-\gamma_{0} t}
\|\xi_{w}(0)\|_{\mathscr{E}^{1}}^{2}
+\mu\int_{0}^{t}e^{-\gamma_{0} s}\|\xi_{w}(s)\|_{\mathscr{E}}^{2}ds
\end{align*}
with $\mu:=\mu(|\Omega|,\mathcal{J}(R_{0}^{2}),J_{0},M_{0},R_{0},R_{1})$. Let $T=\frac{\ln\frac{48}{\eta}}{\gamma_{0}}$, and thereby to find that
\begin{align}\label{ZHOUF5.28}
\|\xi_{w}(T)\|_{\mathscr{E}^{1}}^{2}
\leq\frac{\eta^{2}}{16}\|\xi_{w}(0)\|_{\mathscr{E}^{1}}^{2}
+\mu\int_{0}^{T}\|\xi_{w}(t)\|_{\mathscr{E}}^{2}dt
\end{align}
with $0<\eta<1$. Define the space
\begin{align}\label{ZHOUFE5.29}
\mathds{Z}=\{\xi_{w}=(w,\partial_{t}w)\in L^{2}(0,T;\mathscr{E}^{1})|\partial_{t}^{2}w\in L^{2}(0,T;L^{2})\}
\end{align}
equipped with the norm
\begin{align*}
\|(w,\partial_{t}w)\|_{\mathds{Z}}^{2}=
\int_{0}^{T}(\|\xi_{w}(t)\|_{\mathscr{E}^{1}}^{2}+\|\partial_{t}^{2}w(t)\|^{2})dt.
\end{align*}
Obviously, $\mathds{Z}$ is a Banach space. Let
\begin{align}\label{ZHOUFE5.31}
\textbf{n}_{\mathds{Z}}\left(w,\partial_{t}w\right)
=\sqrt{\mu}\left(\int_{0}^{T}\|\xi_{w}(t)\|_{\mathscr{E}}^{2}dt\right)^{\frac{1}{2}},
\end{align}
we can easily verify $\textbf{n}_{\mathds{Z}}(\cdot)$ defines a compact seminorm on $\mathds{Z}$. Employing \eqref{ZHOU5.21} and estimates \eqref{ZHOU5.26}, we after some computations deduce that
\begin{align}\label{ZHOU5.28}
\|\partial_{t}^{2}w(t)\|\leq C_{12}\|\xi_{w}(t)\|_{\mathscr{E}},\quad\forall t\geq0,
\end{align}
where $C_{12}:=C(|\Omega|,\lambda_{1},M_{0},R_{0},R_{1},\mathcal{J}(R_{0}^{2}))$. Then, define the operator $\mathfrak{C}:\mathds{B}_{1}\rightarrow\mathds{Z}$ by the relation
\begin{align*}
\mathfrak{C}[\xi_{u}(0)](r)=(u(r),\partial_{t}u(r)),\quad r\in[0,T],
\end{align*}
where $u(r)$ is the unique S--S solution of Eq. \eqref{1.1} with initial function $\xi_{u}(0)$. We can rewrite \eqref{ZHOUF5.28} in the following form:
\begin{align}\label{ZHOUF5.31}
\|S(T)\xi_{u}(0)-S(T)\xi_{v}(0)\|_{\mathscr{E}^{1}}
\leq\frac{\eta}{4}\|\xi_{u}(0)-\xi_{v}(0)\|_{\mathscr{E}^{1}}
+\textbf{n}_{\mathds{Z}}(\mathfrak{C}\xi_{u}(0)-\mathfrak{C}\xi_{v}(0)).
\end{align}
On the other hand, using \eqref{ZHOU5.25} and \eqref{ZHOU5.28}, we obtain
\begin{align}\label{ZHOUF5.32}
\nonumber\|\mathfrak{C}\xi_{u}(0)-\mathfrak{C}\xi_{v}(0)\|_{\mathds{Z}}^{2}
&=\int_{0}^{T}(\|\xi_{w}(r)\|_{\mathscr{E}^{1}}^{2}+\|\partial_{t}^{2}w(r)\|^{2})dr
\leq C\int_{0}^{T}\|\xi_{w}(r)\|_{\mathscr{E}^{1}}^{2}dr
\\&\leq CLT\|\xi_{u}(0)-\xi_{v}(0)\|_{\mathscr{E}^{1}}^{2}\xlongequal{\nu=CLT}\nu\|\xi_{u}(0)-\xi_{v}(0)\|_{\mathscr{E}^{1}}^{2}.
\end{align}
Thus, the operator $\mathfrak{C}$ satisfies \eqref{ZHOU5.23}. Consequently, all necessary hypotheses are verified, and the proof is complete.
\end{proof}
\subsection{Dynamics of S--S solutions revisited}
\begin{theorem}\label{ZFT10}
Under Assumption \eqref{A1.1} and the condition $\mathcal{J}(0)>0$, the semigroup $(S(t), \mathscr{E})$ associated with Eq. \eqref{1.1} possesses a global attractor $\mathscr{A}_s$ with the following properties:
\begin{enumerate}[(i)]
\item $\mathscr{A}_{s}$ is compact in $\mathscr{E}^{1}$ and has a finite fractal dimension in $\mathscr{E}^{1}$:
\begin{align*}
\dim_{\mathscr{F}}^{\mathscr{E}^{1}}(\mathscr{A}_{s})<\infty.
\end{align*}
\item $\mathscr{A}_{s}$ is global attracting: for any bounded set $B\subset \mathscr{E}$ it holds that
\begin{align*}
\operatorname{dist}_{\mathscr{E}}\left(S(t)B, \mathscr{A}_{s}\right) \rightarrow 0, \quad \text { as } t \rightarrow \infty .
\end{align*}
\end{enumerate}
\end{theorem}
\begin{proof}
Using Corollary \ref{ZFC5.3}, the global attractor $\mathscr{A}_{s}$ of $(S(t), \mathscr{E})$ established in Theorem \ref{ZFT7} is a bounded set in $\mathscr{E}^{1}$. Since $\mathscr{A}_{s}$ is invariant, it follows that $\mathscr{A}_{s}\subset\mathscr{A}_{exp}^{\iota}(s)$ for any $s\in\mathbb{R}$. Finally, we can estimate the finite fractal dimension of $\mathscr{A}_{s}$ by
\begin{align*}
\operatorname{dim}_\mathscr{F}^{\mathscr{E}}\left(\mathscr{A}_{s}\right)
\leq \operatorname{dim}_\mathscr{F}^{\mathscr{E}^{1}}\left(\mathscr{A}_{s}\right)
\leq \sup_{s\in\mathbb{R}}\dim_{\mathscr{F}}^{\mathscr{E}^{1}}\big(\mathscr{A}_{exp}^{\iota}(s)\big)<\infty .
\end{align*}
Thus, the proof of Theorem \ref{ZFT10} is now complete by using the transitivity of attraction.
\end{proof}
\section{Conclusion}
This paper presents a comprehensive study of the long-term dynamics induced by a wave equation with nonlocal weak damping and quintic nonlinearity in a bounded smooth domain of $\mathbb{R}^3$. The goal is achieved by developing new methodology which allows to circumvent the difficulties related to the lack of compactness and non-locality of the nonlinear damping.

The hypotheses imposed on the damping coefficient allow us to cover a significant class of models featuring nonlocal nonlinear damping terms. We specifically examine the following cases of \eqref{1.1}, where $g$ and $h$ satisfying Assumption \eqref{A1.1} (GH).
\begin{example}
($\mathcal{J}(\|\partial_{t}u(t)\|^{2})\equiv\gamma>0$)
\begin{align}\label{6.1}
\begin{cases}
\partial_{t}^{2}u-\Delta u+\partial_{t}u+g(u)=h(x),
\\u|_{\partial{\Omega}}=0,
\\(u,\partial_{t}u)|_{t=0}=(u^{0},u^{1}).
\end{cases}
\end{align}
The paper \cite{ksz} gives a comprehensive study of long-term dynamics of
of problem \eqref{6.1}. It is easy to see that we can apply the framework introduced in this paper to obtain some similar results constructed in \cite{ksz}.
\end{example}

\begin{example}
Consider the equation
\begin{align}\label{6.2}
\begin{cases}
\partial_{t}^{2}u-\Delta u+\mathcal{J}(\|\partial_{t}u\|^{2})\partial_{t}u+g(u)=h(x),
\\u|_{\partial{\Omega}}=0,
\\(u,\partial_{t}u)|_{t=0}=(u^{0},u^{1}),
\end{cases}
\end{align}
where $\mathcal{J}(s)=\frac{a+s}{b+s}$ (hyperbolic function) or $\mathcal{J}(s)=\frac{ae^{s}}{1+be^{s}}$ (logistic function), where $0<a<b$. Obviously, $\mathcal{J}(\cdot)$ satisfies Assumption \eqref{A1.1} and $\mathcal{J}(0)>0$, and therefore the global attractor with finite dimensionality exists.
\end{example}

\begin{example}
Consider the equation
\begin{align}\label{6.3}
\begin{cases}
\partial_{t}^{2}u-\Delta u+(\|\partial_{t}u\|+\varepsilon)^{p}\partial_{t}u+g(u)=h(x),
\\u|_{\partial{\Omega}}=0,
\\(u,\partial_{t}u)|_{t=0}=(u^{0},u^{1}).
\end{cases}
\end{align}
Here, $\mathcal{J}(s)=(s^{\frac{1}{2}}+\varepsilon)^{p}$, $p>0$ and $0<\varepsilon\ll1$. Then Assumption \eqref{A1.1} is satisfied, and in addition,
$\mathcal{J}(0)=\varepsilon>0$. According to Theorem \ref{ZFT10}, there is thus an attractor $\mathscr{A}$ for the equation \eqref{6.3}, satisfying $\mathscr{A}\Subset\mathscr{E}^{1}$ and $\dim_{\mathscr{F}}^{\mathscr{E}^{1}}(\mathscr{A})<\infty$.
\end{example}

\begin{remark}
The non-degenerate condition $\mathcal{J}(0)>0$ imposed on $\mathcal{J}(\cdot)$ is a crucial element in our analysis of the asymptotic behavior of problem \eqref{1.1}. Define $\mathcal{J}(s)=\mathcal{J}^{1}(s)+\mathcal{J}(0)$, where $\mathcal{J}^{1}(s)=\mathcal{J}(s)-\mathcal{J}(0)$. In contrast to the results in \cite{ksz,zz}, our findings indicate that the linear part $\mathcal{J}(0)$ in the nonlocal coefficient $\mathcal{J}(\cdot)$ plays a master role in the dynamic behavior of the equation, and the non-degeneracy assumption seems necessary for our results.
\end{remark}

\begin{remark}
Recently, the asymptotic behavior of evolution equations with degenerate energy-level damping has been extensively studied. For instance, wave models with various forms of degenerate nonlocal damping have been analyzed, such as $M\left(\int_{\Omega}|\nabla u|^2 d x\right) \partial_t u$ in \cite{cava}, $\left\|\partial_t u\right\|^p \partial_t u$ in \cite{zhongyanzhu,zhong1}, and $\left(\|\nabla u\|^p+\left\|\partial_t u\right\|^p\right) \partial_t u$ in \cite{zhong2}. Similarly, beam models with degenerate nonlocal damping, including $\left(\|\Delta u\|^\theta+q\left\|\partial_t u\right\|^\rho\right)(-\Delta)^\delta \partial_t u$ in \cite{sunzhou}, $k\left(\left\|A^\alpha u\right\|^2+\left\|\partial_t u\right\|^2\right) \partial_t u$ in \cite{gomes}, and $\delta\left(\lambda\|\Delta u\|^2+\left\|\partial_t u\right\|^2+\epsilon I\right)^q \partial_t u$ in \cite{lasiecka}, have also been explored. However, the methods employed in these studies to address the degeneracy of the dissipative term are not directly applicable to the current problem. Consequently, it remains an open question how to obtain a finite-dimensional attractor $\mathscr{A} \Subset \mathscr{E}^1$ for equation \eqref{1.1} when the dissipative term may be degenerate.
\end{remark}

\begin{remark}
The methodology and conceptual framework for constructing attractors presented here could be extended and refined to explore solutions to equations involving nonlinear damping and nonlinear source terms with critical exponents. For instance, wave equations with nonlocal nonlinear damping functions of the form $\sigma\left(\|\nabla u\|^2\right) g\left(\partial_t u\right)$ have been investigated in \cite{sunz}, where the growth exponent $p$ of $g$ is constrained by $1 \leq$ $p<5$. In the critical case where $p=5$, it would be valuable to apply the approach outlined here to reproduce similar results. This topic will be addressed in a forthcoming paper.
\end{remark}

\section*{Data availability}
No data was used for the research described in the article.

\section*{Acknowledgement}
Zhou and Li were supported by the Natural Science Foundation of Shandong Province (ZR2021MA025, ZR2021MA028) and the NSFC of China (Grant No. 11601522); Zhu was supported by the Hunan Province Natural Science Foundation of China (Grant Nos. 2022JJ30417,2024JJ5288); Mei was supported by the NSFC (Grant No. 12201421) and the Yunnan Fundamental Research Project(Grant No. 202401AT070483).



\begin{thebibliography}{99}
{\small
\bibitem{arrieta}J. Arrieta, A. N. Carvalho and J. K. Hale,
\textit{A damped hyperbolic equations with critical exponents}, Comm. Partial Differential Equations, 17 (1992) 841-866.

\vspace{-0.3cm}

\bibitem{aloui}F. Aloui, I.B. Hassen and A. Haraux,
\textit{Compactness of trajectories to some nonlinear second order evolution equations and applications}, J. Math. Pures Appl., 100 (2013) 295--326.

\vspace{-0.3cm}



%

\bibitem{ball} J.M. Ball,
\textit{Global attractors for damped semilinear wave equations}, Discr. Cont. Dyn. Syst. 10 (2004) 31--52.

\vspace{-0.3cm}

\bibitem{bss}M.D. Blair, H.F. Smith and C.D. Sogge,
\textit{Strichartz estimates for the wave equation on manifolds with boundary}, Ann. Inst. Henri Poincar\'{e}, Anal. Non Lin\'{e}aire 26 (2009) 1817--1829.

\vspace{-0.3cm}

\bibitem{caraballo}M.C. Bortolan, T. Caraballo and C. Pecorari Neto,
\textit{Generalized $\varphi$-pullback attractors for evolution processes and application to a nonautonomous wave equation}, Appl. Math. Optim. 89 (2024).

\vspace{-0.3cm}

\bibitem{blp}N. Burq, G. Lebeau and F. Planchon,
\textit{Global existence for energy critical waves in 3-D domains}, J. Am. Math. Soc. 21 (2008) 831--845.

\vspace{-0.3cm}

\bibitem{bp}N. Burq and F. Planchon,
\textit{Global existence for energy critical waves in 3-D domains: Neumann boundary conditions}, Am. J. Math. 131 (2009) 1715--1742.

\vspace{-0.3cm}


%
%
%

\bibitem{carvalho1} A.N. Carvalho and S. Sonner,
\textit{Pullback exponential attractors for evolution processes in Banach spaces: theoretical results}, Commun. Pure Appl. Anal. 12 (2013) 3047--3071.

\vspace{-0.3cm}

\bibitem{carvalho2} A.N. Carvalho and S. Sonner,
\textit{Pullback exponential attractors for evolution processes in Banach spaces: properties and applications}, Commun. Pure Appl. Anal. 13 (2014) 1114--1165.

\vspace{-0.3cm}

%

\bibitem{cava}M.M. Cavalcanti, V.N.D. Cavalcanti, M.A.J. Silva and C.M. Webler,
\textit{Exponential stability for the wave equation with degenerate nonlocal weak damping}, Israel J. Math. 219 (1) (2017) 189--213.


\vspace{-0.3cm}

\bibitem{chang}Q.Q. Chang, D.D. Li, C.Y. Sun and S.V. Zelik,
\textit{Deterministic and random attractors for a wave equation with sign changing damping}, (Russian) Izv. Ross. Akad. Nauk Ser. Mat. 87 (2023) 161--210.

\vspace{-0.3cm}


%

\bibitem{cheskidov}A. Cheskidov,
\textit{Global attractors of evolutionary systems}, J. Dynam. Differ. Equ. 21 (2009) 249--268.

\vspace{-0.3cm}

\bibitem{cheskidov3}A. Cheskidov and L. Kavlie,
\textit{Pullback attractors for generalized evolutionary systems}, Discr. Cont. Dyn. Syst. Ser. B 20 (2015) 749--779.

\vspace{-0.3cm}

\bibitem{cheskidov4}A. Cheskidov and L. Kavlie,
\textit{Degenerate pullback attractors for the 3D Navier-Stokes equations}, J. Math. Fluid Mech. 17 (2015) 411--421.

\vspace{-0.3cm}

\bibitem{cheskidov5}A. Cheskidov and S.S. Lu,
\textit{Uniform global attractors for the nonautonomous 3D Navier-Stokes equations}, Adv. Math. 267 (2014) 277--306.

\vspace{-0.3cm}

\bibitem{chueshov}I. Chueshov,
\textit{Global attractors for a class of Kirchhoff wave models with a structural nonlinear damping}, J. Abstr. Differ. Equ. Appl. 1 (2010) 86--106.

\vspace{-0.3cm}


\bibitem{chueshov00}I. Chueshov,
\textit{Long-time dynamics of Kirchhoff wave models with strong nonlinear damping}, J. Differ. Equ. 252 (2012) 1229--1262.

\vspace{-0.3cm}

\bibitem{chueshov1}I. Chueshov,
\textit{Dynamics of Quasi-Stable Dissipative Systems}, Springer, New York, 2015.



\vspace{-0.3cm}
%
%

\bibitem{chueshov2}I. Chueshov and I. Lasiecka,
\textit{Long-time Behavior of Second Order Evolution Equations with Nonlinear Damping}, Mem. Amer. Math. Soc., vol. 195, 2008, book 912, Providence.

\vspace{-0.3cm}

%
%
%
%

\bibitem{efendiev} M. Efendiev, Y. Yamamoto and A. Yagi,
\textit{Exponential attractors for non-autonomous dissipative systems}, Journal of the Mathematical Society of Japan. 63 (2011) 647--673.
\vspace{-0.3cm}


%
%
%
%
%
%
%
%
%
%
%

\bibitem{lasiecka}E.H. Gomes Tavares, M.A.J. Silva, I. Lasiecka and V. Narciso,
\textit{Dynamics of extensible beams with nonlinear non-compact energy-level damping}, Math. Ann. (2024) Doi: 10.1007/s00208--023--02796--3.

\vspace{-0.3cm}

\bibitem{gomes}E.H. Gomes Tavares, M.A.J. Silva, V. Narciso and V. Vicente,
\textit{Dynamics of a class of extensible beams with degenerate and non-degenerate nonlocal damping},
Adv. in Differ. Equ. 28(7--8) (2023) 685--752.

\vspace{-0.3cm}

%

\bibitem{zhw1}Q.Y. Hu, D.H. Li, S. Liu and H.W. Zhang
\textit{Blow-up of solutions for a wave equation with nonlinear averaged damping and nonlocal nonlinear source terms}, Quaestiones Mathematicae 46(4) (2022) 695--710.

\vspace{-0.3cm}

\bibitem{jorgens}K. J\"{o}rgens,
\textit{Des aufangswert in grossen f\"{u}r eine klasse nichtlinearer wallengleinchungen}, Math. Z. 77 (1961) 295--308.

\vspace{-0.3cm}

%

\bibitem{ksz}V. Kalantarov, A. Savostianov and S. Zelik,
\textit{Attractors for damped quintic wave equations in bounded domains}, Ann. Henri Poincar\'{e}. 17 (2016) 2555--2584.

\vspace{-0.3cm}

%

\bibitem{khanmamedov1}A.K. Khanmamedov,
 \textit{Global attractors for wave equations with nonlinear interior damping
and critical exponents}, J. Differ. Equ. 230 (2006) 702--719.

\vspace{-0.3cm}


%

%
%



\bibitem{yangli}Y.N. Li and Z.J. Yang,
\textit{Optimal attractors of the Kirchhoff wave model with structural nonlinear damping}, J. Differ. Equ. 268 (2020) 7741--7773.

\vspace{-0.3cm}

\bibitem{zhw}D.H. Li, H.W. Zhang and Q.Y. Hu,
\textit{General energy decay of solutions for a wave equation with nonlocal damping and nonlinear boundary damping}, J. Part. Diff. Eq. 32(4) (2019) 369--380.

\vspace{-0.3cm}

\bibitem{lions}J.L. Lions,
\textit{Quelques M\'{e}thodes de R\'{e}solution des Probl\`{e}mes aux Limites Non lin\'{e}aires}, Dunod, Paris, 1969.

\vspace{-0.3cm}

\bibitem{lms}C.C. Liu, F.J. Meng and C.Y. Sun,
\textit{Well-posedness and attractors for a super-cubic weakly damped wave equation with $H^{-1}$ source term}, J. Differ. Equ. 263 (2017) 8718--8748.

\vspace{-0.3cm}

%



%

\bibitem{louredo}A.T. Lour\^{e}do, M. A. Ferreira de Ara\'{u}jo and M.M. Miranda,
\textit{On a nonlinear wave equation with boundary damping}, Math. Methods Appl. Sci., 37(2014) 1278--1302.
\vspace{-0.3cm}

%

%
%
%

\bibitem{mei0}X.Y. Mei and C.Y. Sun,
\textit{Uniform attractors for a weakly damped wave equation with sup-cubic nonlinearity}, Appl. Math. Lett. 95 (2019) 179--185.

\vspace{-0.3cm}

\bibitem{mssz}X.Y. Mei, A. Savostianov, C.Y. Sun and S. Zelik
\textit{Infinite energy solutions for weakly damped quintic wave equations in $\mathbb{R}^{3}$}, Trans. Amer. Math. Soc. 374(5) (2021) 3093--3129.

\vspace{-0.3cm}

%

\bibitem{zp}Q.Q. Peng and Z.F. Zhang,
\textit{Global attractor for a coupled wave and plate equation with nonlocal weak damping on Riemannian manifolds}, Appl. Math. Optim. 88 (2023) 28.

\vspace{-0.3cm}

%

%
%
%
%


%

\bibitem{sav}A. Savostianov,
\textit{Strichartz estimates and smooth attractors for a sub-quintic wave equation with fractional damping in bounded domains}, Adv. Differ. Equ. 20 (2015) 495--530.

\vspace{-0.3cm}

\bibitem{sav1}A. Savostianov,
\textit{Strichartz Estimates and Smooth Attractors of Dissipative Hyperbolic Equations}, (Doctoral dissertation), University of Surrey, 2015.

\vspace{-0.3cm}

\bibitem{sav2}A. Savostianov and S. Zelik,
\textit{Uniform attractors for measure-driven quintic wave
equation with periodic boundary conditions}, Uspekhi Mat. Nauk vol. 75 2(452) (2020) 61--132.

\vspace{-0.3cm}

\bibitem{schiff}L. Schiff,
\textit{Nonlinear meson theory of nuclear forces I. Neutral scalar mesons with point-contact repulsion}, The Physical Reviews. Second Series, 84(1) (1951) 1--9.

\vspace{-0.3cm}

%
%
%
%
%
%

%
\bibitem{sun}C.Y. Sun, D.M. Cao and J.Q. Duan,
\textit{Uniform attractors for non-autonomous wave equations with nonlinear damping}, SIAM J. Appl. Dynam. Syst. 6 (2007) 293--318.

\vspace{-0.3cm}
%

%

%

\bibitem{zhong2}Z.J. Tang, S.L. Yan, Y. Xu and C.K. Zhong,
\textit{Finite-dimensionality of attractors for wave equations with degenerate nonlocal damping}, Discr. Cont. Dyn. Syst. (2024) Doi: 10.3934/dcds.2024091.

\vspace{-0.3cm}

\bibitem{temam}R. Temam,
\textit{Infinite-Dimensional Dynamical Systems in Mechanics and Physics}, Second edition. Applied Mathematical Sciences, Springer-Verlag, New York, 1997.

\vspace{-0.3cm}


\bibitem{sunxiong}Y.M. Xiong and C.Y. Sun,
\textit{Kolmogorov $\epsilon$-entropy of the uniform attractor for a wave equation}, J. Differential Equations 387 (2024) 532--554.

\vspace{-0.3cm}

\bibitem{zhongyt}S.L. Yan, Z.J. Tang and C.K. Zhong,
\textit{Strong attractors for weakly damped quintic wave equation in bounded domains}, J. Math. Anal. Appl. 519 (2023).

\vspace{-0.3cm}

\bibitem{zhongyanzhu}S.L. Yan, X.M. Zhu, C.K. Zhong and Z.J. Tang,
\textit{Long-time dynamics of the wave equation with nonlocal weak damping and super-cubic nonlinearity in 3-d domains, part II: Nonautonomous case}, Appl. Math. Optim. 88(3) (2023) 1--38.

\vspace{-0.3cm}

%
%
%
\bibitem{zelik}S. Zelik,
\textit{Asymptotic regularity of solutions of a nonautonomous damped wave equation with a critical growth exponent}, Comm. Pure Appl. Anal. 3
 (2004) 921--934.

\vspace{-0.3cm}

%

\bibitem{zelik1}S. Zelik,
\textit{Asymptotic regularity of solutions of singularly perturbed damped wave equations with supercritical nonlinearities}, Discrete Contin. Dyn. Syst. 11(2--3) (2004) 351--392.
\vspace{-0.3cm}

\bibitem{zhong}C.Y. Zhao, C.X. Zhao and C.K. Zhong,
\textit{Asymptotic behavior of the wave equation with nonlocal weak damping and anti-damping}, J. Math. Anal. Appl. 490 (2020).

\vspace{-0.3cm}

%
\bibitem{zzt}C.Y. Zhao, C.K. Zhong and Z.J. Tang,
\textit{Asymptotic behavior of the wave equation with nonlocal weak damping, anti-damping and critical nonlinearity}, Evol. Equ. Control Theory 12 (2023) 154--174.

\vspace{-0.3cm}

\bibitem{zhong1}C.Y. Zhao, C.K. Zhong and S.L. Yan,
\textit{Existence of a generalized polynomial attractor for the wave equation with nonlocal weak damping, anti-damping and critical nonlinearity}, Appl. Math. Lett. 128 (2022).

\vspace{-0.3cm}

%

\bibitem{zzhaoz}C.Y. Zhao, C.K. Zhong and X.M. Zhu,
\textit{Existence of compact $\varphi$-attracting sets and estimate of their attractive velocity for infinite-dimensional dynamical systems}, Discrete Contin. Dyn. Syst. Ser. B 27 (2022) 7493--7520.

\vspace{-0.3cm}

%

\bibitem{sunz}C. Zhou and C.Y. Sun,
\textit{Global attractor for a wave model with nonlocal nonlinear damping}, J. Math. Anal. Appl. 507 (2022) 125818.

\vspace{-0.3cm}

\bibitem{sunzhou}C. Zhou and C.Y. Sun,
\textit{Stability for a class of extensible beams with degenerate nonlocal damping}, J. Geom. Anal. 33 295 (2023).

\vspace{-0.3cm}

\bibitem{zhou}F. Zhou, Z.Y. Sun, K.X. Zhu and X.Y. Mei,
\textit{Exponential Attractors for the Sup-Cubic Wave Equation with Nonlocal Damping}, Bull. Malays. Math. Sci. Soc. 47 (2024) Paper No. 104.

\vspace{-0.3cm}

\bibitem{zz}X.M. Zhu, and C.K. Zhong,
\textit{A note on the polynomially attracting sets for dynamical systems}, J. Dynam. Differential Equations 36 (2024) 1873--1878.

}
\end{thebibliography}
\end{document}